\documentclass{amsart}     
\usepackage{graphicx}

\usepackage{amsmath, amssymb}
\newtheorem{theorem}{Theorem}
\newtheorem{lemma}[theorem]{Lemma}

\newcommand{\be}{\begin{equation}}
\newcommand{\ee}{\end{equation}}
\newcommand{\threevec}[3]{\left( \begin{array}{c}#1\\#2\\#3\end{array} \right)}
\newcommand{\twovec}[2]{\left( \begin{array}{c}#1\\#2\end{array} \right)}

\newcommand{\uv}{\mathbf{u}}
\newcommand{\thetav}{\boldsymbol{\theta}}
\newcommand{\etav}{\boldsymbol{\eta}}
\newcommand{\gammav}{\boldsymbol{\gamma}}
\newcommand{\deltav}{\boldsymbol{\delta}}
\newcommand{\sigmav}{\boldsymbol{\sigma}}
\newcommand{\tauv}{{\boldsymbol{\tau}}}
\newcommand{\nv}{\mathbf{n}}
\newcommand{\vv}{\mathbf{v}}
\newcommand{\fv}{\mathbf{f}}
\newcommand{\gv}{\mathbf{g}}

\newcommand{\sigmat}{\underline{\boldsymbol{\sigma}}}
\newcommand{\taut}{\underline{\boldsymbol{\tau}}}
\newcommand{\epst}{\underline{\boldsymbol{\varepsilon}}}
\newcommand{\Mt}{\underline{\mathbf{m}}}
\newcommand{\Mv}{\mathbf{m}}
\newcommand{\M}{m}
\newcommand{\Ct}{\underline{\mathbf{C}}_b}
\newcommand{\At}{\underline{\mathbf{A}}_b}
\newcommand{\Ctel}{\underline{\mathbf{C}}}
\newcommand{\Atel}{\underline{\mathbf{A}}}

\newcommand{\tang}{{\boldsymbol{\tau}}}
\newcommand{\Triang}{\mathcal{T}}
\newcommand{\Facets}{\mathcal{F}}

\newcommand{\spaceTheta}{\mathbf{\Theta}}
\newcommand{\spaceM}{\underline{\mathbf{M}}}
\newcommand{\spaceGamma}{\mathbf{\Gamma}}
\newcommand{\spaceW}{W}

\newcommand{\opB}{B}
\newcommand{\bilA}{a}
\newcommand{\bilB}{b}

\newcommand{\opdiv}{\operatorname{div}}
\newcommand{\opdivv}{\mathbf{{div}}}
\newcommand{\opcurl}{\operatorname{curl}}
\newcommand{\opker}{\operatorname{Ker}}

\newcommand{\IntTheta}{\mathcal{I}_{\spaceTheta}}
\newcommand{\IntW}{\mathcal{I}_{\spaceW}}
\newcommand{\IntSigma}{\mathcal{I}_{\spaceM}}

\newcommand{\Lv}{{\mathbf{L}}}
\newcommand{\Hv}{{\mathbf{H}}}
\newcommand{\Lt}{{\underline{\mathbf{L}}}}
\newcommand{\Ht}{{\underline{\mathbf{H}}}}
\newcommand{\HcurlzeroOmega}{{\mathbf{H}_0(\opcurl,\Omega)}}
\newcommand{\Hcurlzero}{{\mathbf{H}_0(\opcurl)}}
\newcommand{\Hcurl}{{\mathbf{H}(\opcurl)}}
\newcommand{\HcurlOmega}{{\mathbf{H}(\opcurl,\Omega)}}
\newcommand{\LtwotwoOmega}{{\mathbf{L}^2(\Omega)}}
\newcommand{\LtwosymOmega}{{\underline{\mathbf{L}}^2_{sym}(\Omega)}}

\newcommand{\Hdivdiv}{{\underline{\mathbf{H}}(\opdiv\opdivv)}}
\newcommand{\HmonedivOmega}{{\mathbf{H}^{-1}(\opdiv,\Omega)}}
\newcommand{\Hmonediv}{{\mathbf{H}^{-1}(\opdiv)}}

\newcommand{\red}[1]{#1}

\begin{document}


\title{The TDNNS method for Reissner-Mindlin plates}

\author{Astrid S.\ Pechstein}
\address{Astrid. S.\ Pechstein\\ Institute of Technical Mechanics\\ Johannes Kepler University Linz\\ Altenbergerstr. 69\\ 4040 Linz, Austria}
\email{astrid.pechstein@jku.at}
\author{Joachim\ Sch\"oberl}
\address{Joachim\ Sch\"oberl\\ Institute for Analysis and Scientific Computing\\ Vienna University of Technology\\ Wiedner Hauptstrasse 8-10\\ 1040 Wien, Austria}
\email{joachim.schoeberl@tuwien.ac.at}

\date{\today}

\begin{abstract}
A new family of locking-free finite elements for shear deformable Reissner-Mindlin plates is presented. The elements are based on the ``tan\-gen\-tial-displacement normal-normal-stress'' formulation of elasticity. In this formulation, the bending moments are treated as separate unknowns. The degrees of freedom for the plate element are the nodal values of the deflection, tangential components of the rotations and normal-normal components of the bending strain.
Contrary to other plate bending elements, no special treatment for the shear term such as reduced integration is necessary. The elements attain an optimal order of convergence.

\end{abstract}

\maketitle

\section{Introduction} \label{sec:intro}

In this paper we are concerned with finite elements for shear deformable plates based on the Reissner-Mindlin model \cite{Reissner:45,Mindlin:51}.
A direct discretization of the equations leads to shear locking phenomena as the plate thickness becomes small. In the limit of zero thickness, the Kirchhoff assumption is enforced, where the shear strain vanishes and the deflection gradient equals the rotations.
Over the last decades, a vast amount of different elements overcoming shear locking by different kinds of remedies has been proposed.
In most standard, conforming finite element methods, the Kirchhoff constraint of vanishing shear stress is alleviated or modified in some way. An alternative are discontinuous Galerkin (DG) methods, mixed, or hybrid methods. 

An example for the alleviation of the Kirchhoff constraint is the ``assumed shear strain method'' introduced by MacNeal \cite{MacNeal:78}. A special operator for the displacement-strain relation relying on local averaging is used, this approach was further developed by \cite{HughesTezduyar:81,BatheDvorkin:85,CastellazziKrysl:09}.
In a further method referred to as ``linked interpolation'', the displacement gradient in the Kirchhoff constraint is augmented by a ``kinematic linking operator''. Pioneers in this field were Zienkiewicz and co-authors \cite{ZienkiewiczEtAl:93} and Taylor and Auricchio \cite{TaylorAuricchio:93}. In \cite{AuricchioLovadina:01}, Auricchio and Lovadina provide an analysis of general linked-interpolation elements.

Other methods employ additional unknowns for the shear stress quantity, which allows to pose the Kirchhoff condition of vanishing shear in weak sense. Examples are the Falk-Tu element \cite{FalkTu:99} or the MITC element \cite{BrezziBatheFortin:89}. In implementations, the further unknown can be eliminated element-wise, which leads to a projection of the shear stresses in the penalty term. This projection is referred to as ``reduction'' in \cite{BrezziBatheFortin:89}. In some cases, it may be achieved by reduced integration of the shear term. This approach is analyzed minutely for the one-dimensional case of a thick beam in \cite{Arnold:81}. Reduced integration is also used without the background of an additional shear stress unknown \cite{ZienkiewiczTaylor:71,GruttmannWagner:04}, where care has to be taken to avoid spurious modes.

For MITC elements, error analysis has been provided; see e.g.\ the works by Brezzi, Fortin and Stenberg \cite{BrezziFortinStenberg:91}, where additionally a postprocessing step for the deflections is proposed, or by Stenberg and Suri \cite{StenbergSuri:97} for an $hp$ error analysis.

Nonconforming elements have been constructed by Arnold and Falk \cite{ArnoldFalk:89}, where Crouzeix-Raviart elements are used for the deflection.
More recently, Brezzi and Marini \cite{BrezziMarini:03} developed a nonconforming element in the framework of discontinuous Galerkin methods. In both works, the shear strain is projected into a lower-order finite element space to alleviate the Kirchhoff constraint.
Other DG approaches allow for a direct enforcement of the Kirchhoff constraint, as the rotation space can be chosen such that it contains the deflection gradient. Deflections, rotations and the shear are approximated using different continuity assumptions in \cite{ArnoldBrezziFalkMarini:07}. In \cite{BoesingMadureiraMozolevski:10,HansboHeintzLarson:11} DG methods for deflection and rotation without reduced integration techniques are presented. We also mention \cite{CaloCollierNiemi:14} for a discontinuous Petrov-Galerkin method, where optimal test functions of higher polynomial order are chosen to suit the trial functions. A quadrilateral hybrid finite element method was introduced in \cite{CarstensenXieYuZhou:11}, where shear stress and bending moment are discontinuous and the corresponding finite element basis is constructed to satisfy a local equilibrium condition.

Hughes and Franca \cite{HughesFranca:88} added an additional stabilization term to the variational equations.
Also, Chapelle and Stenberg \cite{ChapelleStenberg:98} augmented the equations by a stabilization term, which then allows for an analysis ensuring an optimal order of convergence.
An entirely different approach by Pontaza and Reddy \cite{PontazaReddy:04} is to use a least squares method instead of the standard Galerkin equations.

Mixed method with weak symmetry for the tensor of bending moments are proposed in \cite{VeigaMoraRodriguez:13,BehrensGuzman:11}. In both works, the bending moments are approximated in the normal-continuous space $\Ht(\opdivv)$. Since symmetric $\Ht(\opdivv)$-conforming elements are hard to construct and of high polynomial order \cite{ArnoldWinther:02}, imposing symmetry weakly using a Lagrange multiplier has often been proposed in the context of continuum mechanics, see e.g.\ \cite{Stenberg:86,Stenberg:88,ArnoldBrezziDouglas:84,ArnoldFalkWinther:07}. The continuum mechanics formulation used in the current paper overcomes this problem. 

The plate elements proposed in the current paper are based on the ``tan\-gen\-tial-displacement normal-normal-stress'' (TDNNS) formulation of elasticity introduced by the authors in \cite{PechsteinSchoeberl:11}. This leads to a formulation containing deflection, rotations and bending moments as separate unknowns. While the deflection is sought in the standard Sobolev space $H^1$, the rotation is assumed to be in the less regular space $\Hcurl$. Additionally the TDNNS stress space $\Hdivdiv$ is chosen for the bending moments. Accordingly, we use standard continuous finite elements for the deflection, tangential continuous N\'ed\'elec elements for the rotations, and normal-normal continuous tensor valued elements from \cite{PechsteinSchoeberl:11} for the bending moments.
The main benefit of this choice is that the gradient of the deflection space is a subset of the rotation space both for the infinite dimensional and for the finite element problem, and the Kirchhoff constraint of vanishing shear strain does not lead to locking. Thus, the proposed formulation seems to be a very natural alternative to the established ones based on $H^1$ continuity and continuous finite elements. In \cite{PechsteinSchoeberl:11anis} the authors have shown that three-dimensional anisotropic TDNNS elements are suitable for the discretization of slim domains.
The proposed Reissner-Mindlin plate elements show an optimal order of convergence, which is confirmed in our numerical results.

\red{
The proposed elements are closely related to the Hellan-Herrmann-Johnson element for the  bending problem of a Kirchhoff plate \cite{Hellan:67,Herrmann:67,Johnson:73}. Also in the HHJ formulation, the normal-normal component of the bending moment is continuous across interfaces. As the Hellan-Herrmann-Johnson element is restricted to the biharmonic problem, the rotations are not treated independently as done in the current work. On the other hand, for vanishing thickness one can eliminate the rotations in the current formulation, and arrives at the HHJ plate formulation. Thus, the HHJ formulation may be seen as the limiting case of the proposed method. In \cite{Comodi:89}, Lagrangian multipliers for the normal-normal component of the bending moment are introduced. Postprocessing then leads to a faster convergence of the deflection gradient.
} 


This work is organized as follows: in Section~\ref{sec_problemformulation}, the TDNNS method is shortly introduced and applied to the Reissner-Mindlin problem. An analysis of the infinite dimensional problem using the TDNNS spaces for positive as well as vanishing thickness is provided in Section~\ref{sec:analysis}. Finite elements are introduced in Section~\ref{sec:fem}, and a-priori error estimates are provided. We mention hybridization of the bending moments by Lagrangian multipliers resembling the normal component of the rotation, which results after static condensation in a symmetric positive system matrix.  Finally, Section~\ref{sec:numerics} contains numerical examples verifying the claimed convergence orders.

\section{Problem formulation} \label{sec_problemformulation}

\subsection{Notation}

In the following, all vectors are denoted by boldface letters, tensors are boldface and underlined.
Let $\Omega \subset \mathbb R^2$ be a bounded, connected, polygonal Lipschitz domain. Its unit outward normal $\nv$ is defined almost everywhere on the boundary $\partial \Omega$. The unit tangential vector in counter-clockwise direction $\tauv$ is given as the rotation of $\nv$,
\begin{equation}
    \tauv = \nv^\bot = \twovec{-n_y}{n_x}.
\end{equation}
For a vector field $\vv$ on $\Omega$, normal and tangential component on the boundary are denoted by
\begin{equation}
    v_n := \vv \cdot \nv,\qquad  \vv_\tauv = \vv - v_n \nv.
\end{equation}
For a tensor field of second order $\sigmat$, its normal component is given by
$   \sigmav_\nv := \sigmat \nv$. %
The normal component $\sigmav_\nv$ can further be split into a (scalar-valued) normal-normal component $\sigma_{nn}$ and a (vector-valued) normal-tangential component $\sigmav_{\nv\tauv}$,
\begin{equation}
    \sigma_{nn} = \sigmav_\nv \cdot \nv, \qquad \sigmav_{\nv\tauv} = \sigmav_\nv - \sigma_{nn} \nv.
\end{equation}

Rotation and divergence of a two-dimensional vector field shall be denoted by $\opcurl$, $\opdiv$, respectively. The operator $\opdivv$ is the row-wise divergence operator, mapping tensor to vector fields.

For a general Hilbert space $V$, its inner product and norm are denoted by $(\cdot,\cdot)_V$ and $\|\cdot\|_V$, respectively.
The duality product between $V$ and its dual $V^*$ is denoted by angles,
\begin{equation}
    \langle f, g \rangle = f(g) \qquad \forall f \in V^*, g \in V.
\end{equation}

Let $L^2(\Omega)$ denote the usual Lebesgue space. Moreover, let $H^1(\Omega)$  be the usual Sobolev space of weakly differentiable functions, and let $H^1_0(\Omega)$ be the space of $H^1$ functions satisfying zero boundary conditions.
We also use the Sobolev spaces $H^s(\Omega)$ for integer $s$.
 Moreover, $\HcurlOmega$ and $\HcurlzeroOmega$ shall be the spaces of vector-valued functions with weak rotation, the latter satisfying zero boundary conditions for the tangential component of the vector fields, see \cite{Monk:03}.

The dual space of $H^1_0(\Omega)$ shall be denoted by $H^{-1}(\Omega)$. It is well established \cite[Equation~(10.4.52)]{BoffiBrezziFortin:13} that the dual space of $\HcurlzeroOmega$ is $\HmonedivOmega$ being the space of $\mathbf{H}^{-1}$ vector fields with distributional divergence in $H^{-1}$,
\begin{equation} \label{eq:defhmonediv}
    \HmonedivOmega = \{ \fv \in \mathbf{H}^{-1}(\Omega): \opdiv \fv \in H^{-1}(\Omega)\} = (\HcurlzeroOmega)^* .
\end{equation}
The distributional divergence operator is defined by the relationship
\begin{equation}
 \langle \opdiv \fv, w\rangle = -\langle \fv, \nabla w\rangle \qquad \forall w \in C^\infty_0(\Omega).
\end{equation}
Using this definition, a natural norm of $\HmonedivOmega$ is
\begin{eqnarray}
	\|\fv\|_{\HmonedivOmega}^2 &=& \|\fv\|_{\Hv^{-1}(\Omega)}^2 + \|\opdiv\fv\|_{H^{-1}(\Omega)}^2 \\
	&=&	\sup_{\vv \in \Hv^1_0(\Omega)} \frac{\langle \fv, \vv \rangle^2}{\|\nabla \vv\|_{\Lt^2(\Omega)}^2} + \sup_{w \in H^1_0(\Omega)} \frac{\langle \fv, \nabla w\rangle^2}{\|\nabla w\|_{\Lv^2(\Omega)}^2}.
	\label{eq:8}
\end{eqnarray}
In all further occurrences of Sobolev spaces on $\Omega$, the domain can also be omitted. We thus write $H^1$ for $H^1(\Omega)$ or $\Hcurl$ for $\HcurlOmega$. When defined on domains other than $\Omega$, the domain must be indicated as above.

%

\subsection{The TDNNS method} \label{sec:TDNNS}

In \cite{PechsteinSchoeberl:11,PechsteinSchoeberl:11anis} we introduced the tangential-displacement normal-normal-stress (TDNNS) method for elasticity, and refined the analysis in \cite{PechsteinSchoeberl:NN}. In the current section, we briefly cover the main idea of the TDNNS method for elasticity problems in the two-dimensional continuum, as the proposed Reissner-Mindlin elements will be based on this method.

Let $\Omega \subset \mathbb R^2$ be a bounded, connected, polygonal domain
with Lipschitz boundary $\partial \Omega$. The displacement vector
$\uv = (u_1,u_2)$ and symmetric stress tensor $\sigmat \in \mathbb
R^{2\times2}_{sym}$ are connected via Hooke's law
\eqref{eq:hooke} and the equilibrium equation
\eqref{eq:equilibrium}
\begin{align}
    \Atel \sigmat &= \epst(\uv) && \mbox{in } \Omega, \label{eq:hooke}\\
    \opdivv \sigmat &= -\fv && \mbox{in } \Omega.  \label{eq:equilibrium}
\end{align}
Here, we use the linearized strain tensor $\epst(\uv) = \frac{1}{2} (\nabla \uv + (\nabla \uv)^T)$. Stress and strain are connected via the compliance tensor $\Atel$, which is the inverse of the elasticity tensor $\Ctel$ depending on Young's modulus $E$ and the Poisson ratio $\nu$ in the well known way. For simplicity, we assume that homogeneous displacement boundary conditions are posed on $\partial \Omega$,
\begin{equation}
    \uv = 0 \qquad \mbox{on } \partial \Omega.
\label{eq:bc_elasticity}
\end{equation}

Most standard methods for the elasticity problem rely on a primal formulation, which is obtained eliminating the stress tensor $\sigmat$ from equations \eqref{eq:hooke} and \eqref{eq:equilibrium}. Then one searches for the displacement $\uv$ in $\mathbf{H}^1_0$ such that
\begin{equation}
    \int_\Omega \Ctel \epst(\uv):\epst(\vv)\,dx = \int_\Omega \fv \cdot \vv\,dx \qquad \forall \vv \in \mathbf{H}^1_0.
\label{eq:elasticity_primal}
\end{equation}
In a conforming finite element method, the displacement vector $\uv$ is approximated by a continuous finite element function.

On the other hand, the dual Hellinger-Reissner formulation is obtained directly from system \eqref{eq:hooke}, \eqref{eq:equilibrium} when multiplying with test functions and using integration by parts,
\begin{align}
\int_\Omega \Atel \sigmat: \taut\, dx + \int_\Omega \opdivv \taut \cdot \uv\, dx &= 0 && \forall \taut \in  \underline{\mathbf{H}}_{sym}(\opdivv), \label{eq:dualprob1}\\
\int_\Omega \opdivv \sigmat \cdot \vv\, dx &=
\int_\Omega \fv \cdot \vv\, dx && \forall \vv \in \mathbf{L}^2.
\label{eq:dualprob2}
\end{align}
Here, the solution spaces are $\underline{\mathbf{H}}_{sym}(\opdivv)$ for the stress and $\mathbf{L}^2$ for the displacement. In a conforming finite element method, the displacement elements can be totally discontinuous, while the stress elements need to be tensor-valued symmetric and normal-continuous. These requests lead to computationally expensive finite elements of at least 24 degrees of freedom in two dimensions (see \cite{ArnoldWinther:02}) and 162 degrees of freedom in three dimensions (see \cite{AdamsCockburn:05,ArnoldAwanouWinther:07}).

The TDNNS formulation is ``in between'' those two concepts, where the displacements are not assumed totally continuous or discontinuous, but where the tangential component is assumed to be continuous across element borders. In a mathematical setting, the displacement space is chosen as
\begin{equation}
     \Hcurlzero = \{ \vv \in \Hcurl: \vv_\tang = 0 \mbox{ on } \partial \Omega\}.
\end{equation}
In \cite{PechsteinSchoeberl:11,PechsteinSchoeberl:NN} we have shown that the corresponding stress space is the space of symmetric $\Lt^2$ tensors with weak divergence in the dual space of $\Hcurlzero$. Due to \eqref{eq:defhmonediv}, the stress space is given by%
\begin{equation}
     \Hdivdiv = \{\taut \in \Lt^2_{sym}: \opdiv\opdivv \taut \in H^{-1}\}.
\end{equation}
The duality product $\langle\opdivv\taut, \vv\rangle$, where the divergence of a stress tensor is applied to a displacement field, plays an important role in the TDNNS method.
In \cite{PechsteinSchoeberl:NN} we elaborated on the meaning of the duality product $\langle\opdivv\taut, \vv\rangle$ for $\taut \in \Hdivdiv$ and $\vv \in \Hcurlzero$. \red{We state that, for smooth $\taut \in \Hdivdiv$ and $\vv \in \Hcurlzero$, i.e.\ for $\vv$ smooth with vanishing tangential component $\vv_\tang = 0$ on $\partial \Omega$, the duality product can be evaluated by}
\begin{equation}
\langle\opdivv\taut, \vv\rangle = -\int_\Omega \taut : \epst(\vv)\,dx + \int_{\partial \Omega} \tau_{nn} v_n\, ds. \label{eq:def_dualityproduct}
\end{equation}
A natural norm of the stress space $\Hdivdiv$ uses this duality product and is given by, see \cite{PechsteinSchoeberl:NN}
\begin{equation}
	\|\taut\|_{\Hdivdiv}^2 = \|\taut\|_{\Lt^2}^2 + \|\opdiv\opdivv\taut\|_{H^{-1}}^2 = \|\taut\|_{\Lt^2}^2 + \sup_{w \in H^1_0 \cap H^2} \frac{\langle \opdivv \taut, \nabla w\rangle^2}{\|\nabla w\|_{\Lv^2}^2}. \label{eq:norm_hdivdiv}
\end{equation}

It is well known that finite elements for $\Hcurl$ have to be tangential continuous, such as N\'ed\'elec elements introduced in \cite{Nedelec:80,Nedelec:86}. In \cite{PechsteinSchoeberl:11} it was shown that finite elements for the stress space $\Hdivdiv$ are normal-normal continuous, meaning that the normal component $\sigma_{nn}$ of the normal stress vector $\sigmav_\nv$ is continuous across element borders.

One obtains a variational problem of a form similar to the dual problem \eqref{eq:dualprob1} - \eqref{eq:dualprob2},
\begin{align}
    \int_\Omega \Atel \sigmat: \taut\,dx + \langle\opdivv \taut, \uv\rangle &= 0 && \forall \taut \in \Hdivdiv, \label{eq:hdd1}\\
    \langle \opdivv \sigmat, \vv\rangle &= \int_\Omega \fv\cdot \vv\,dx && \forall \vv \in \Hcurlzero. \label{eq:hdd2}
\end{align}
In a finite element method, it is necessary to evaluate duality products of the form $\langle \opdivv\taut, \vv\rangle$ for piecewise smooth functions with $\tau_{nn}$ and $\vv_\tang$ continuous on a finite element mesh $\mathcal T = \{T\}$. In this case, the definition from \eqref{eq:def_dualityproduct} can be extended to
\begin{eqnarray}
    \langle \opdivv\taut, \vv\rangle &=& \sum_{T \in \mathcal T} \Big( \int_T \opdivv \taut \cdot \vv\,dx - \int_{\partial T} \tauv_{\nv\tang} \cdot \vv_\tang\,ds \Big) \label{eq:def_div1}\\
    &=& -\sum_{T \in \mathcal T} \Big( \int_T \taut : \epst(\vv)\,dx - \int_{\partial T} \tau_{nn} \cdot v_n\,ds
    \Big), \label{eq:def_div2}
\end{eqnarray}
%
It was shown (see \cite{PechsteinSchoeberl:11,PechsteinSchoeberl:NN}) that the infinite dimensional problem \eqref{eq:hdd1}, \eqref{eq:hdd2} is well posed. Moreover, a stable family of mixed finite elements was constructed, using N\'ed\'elec's elements for the displacement space and a new family of tensor-valued symmetric normal-normal continuous elements for the stress space. The two-dimensional mixed finite elements shall be used in the Reissner-Mindlin elements proposed in this work.

%

\subsection{Reissner-Mindlin model}

Let again $\Omega \subset \mathbb R^2$ be a bounded, connected domain with Lipschitz boundary $\partial \Omega$. We consider a plate of thickness $t$ corresponding to the three-dimensional domain $\Omega \times (-t/2,t/2)$. In the Reissner-Mindlin model, the displacement vector $\uv$ is assumed to take the form
\be
    \uv = \threevec{-z\theta_1}{-z\theta_2}{w},
\ee
where $\thetav = (\theta_1, \theta_2)$ are rotations and $w$ is the deflection in vertical $z$ direction. Both the rotations $\thetav$ and the deflection $w$ are assumed to depend on the in-plane coordinates $(x_1, x_2)$ only.

Assuming a vertical volume load $\fv = (0,0,t^2 g)^T \in
\Lv^2$ to be given, the Reissner-Mindlin
problem for a clamped plate is to find the deflection $w$ and
rotations $\thetav$ such that
\begin{align}
    -\opdivv( \Ct \epst(\thetav)) - \mu t^{-2}(\nabla w - \thetav) & =  0 & \mbox{in } \Omega, \label{eq:RM1}\\
    -\mu t^{-2}\opdiv (\nabla w - \thetav) & =  g  & \mbox{in } \Omega, \label{eq:RM2}\\
    \thetav &= 0 & \mbox{on } \partial \Omega, \label{eq:RM3}\\
    w &= 0 & \mbox{on } \partial \Omega. \label{eq:RM4}
\end{align}
Here $\Ct$ is the tensor of bending moduli and $\mu$ is the shear modulus with shear correction factor $k_s$, which depend on Young's modulus $E$ and Poisson's ratio $\nu$ via
\begin{equation}
    \Ct = \frac{E}{12(1-\nu^2)}\left( \begin{array}{ccc} 1 & \nu & 0 \\ \nu & 1 & 0 \\ 0 & 0 & \frac{1-\nu}{2} \end{array} \right),\qquad
    \mu = \frac{k_s E}{2(1+\nu)}.
\end{equation}
Additionally, we provide the compliance tensor $\At = \Ct^{-1}$
\begin{equation}
    \At = \Ct^{-1} = \frac{12}{E}\left( \begin{array}{ccc} 1 & -\nu & 0 \\ -\nu & 1 & 0 \\ 0 & 0 & 2(1+\nu) \end{array} \right).
\end{equation}

\red{
Of course, other boundary conditions such as simply supported or free boundaries, or boundary tractions and moments, may be prescribed. Although the analysis of the proposed finite element formulation is done for the clamped case for sake of simplicity, we shall comment shortly on the implementation of other boundary conditions in the end of the current section. We will see that all common types of boundary conditions can be treated in a very natural way.
}

As in the continuum problem in Section~\ref{sec:TDNNS}, a primal, displacement-based variational formulation of the Reissner-Mindlin problem \eqref{eq:RM1}-\eqref{eq:RM4} can be obtained directly. Both the rotations $\thetav$ and the deflection $w$ are assumed weakly differentiable, with $\thetav \in \spaceTheta_{prim} = \mathbf{H}^1_0$ and $w \in \spaceW_{prim} = H^1_0$. The primal variational formulation of the Reissner-Mindlin problem is to find $\thetav \in \spaceTheta_{prim}$ and $w \in \spaceW_{prim}$ such that for all test functions $\etav \in \spaceTheta_{prim}$ and $v \in \spaceW_{prim}$
\begin{eqnarray}
    \int_\Omega \Ct \epst(\thetav):\epst(\etav)\,dx + \mu t^{-2}\int_\Omega (\nabla w - \thetav)\cdot(\nabla v - \etav)\,dx &=& \int_\Omega gv\,dx .
\end{eqnarray}
A straightforward finite element discretization of this primal problem choosing continuous finite element spaces $\spaceTheta_{prim,h} \subset \spaceTheta_{prim}$ and $\spaceW_{prim,h} \subset \spaceW_{prim}$ leads to shear locking as the thickness $t$ tends to zero. In the limit case of a Kirchhoff plate with $t=0$, the condition of vanishing shear strain
\begin{equation}
    \nabla w - \thetav = 0 \label{eq:kirchhoff}
\end{equation}
has to be satisfied. For conventional finite element discretizations one observes that $\nabla \spaceW_{prim,h} \not\subset \spaceTheta_{prim,h}$, thus the Kirchhoff constraint \eqref{eq:kirchhoff} cannot be satisfied by the discrete solution, the formulation locks. Different methods have been proposed to reduce this phenomenon by alleviating the Kirchhoff constraint \eqref{eq:kirchhoff}, see Section~\ref{sec:intro}.
In this work, the rotation space will be chosen such that both $\nabla \spaceW \subset \spaceTheta$ and $\nabla \spaceW_h \subset \spaceTheta_h$. This ensures stability and an optimal order of convergence of the method.

The main idea of the current work is to use the TDNNS method presented in Section~\ref{sec:TDNNS} for the discretization of rotations $\thetav$. To this end, additional unknowns $\Mt = \Ct \epst(\thetav)$ for the tensor of bending moments are introduced. This leads to the following system of partial differential equations
\begin{align}
    \At \Mt - \epst(\thetav) & = 0 &\mbox{in } \Omega \label{eq:RMmixed1}, \\
    -\opdivv( \Mt) - \mu t^{-2}(\nabla w - \thetav) & =  0 & \mbox{in } \Omega, \label{eq:RMmixed2}\\
    -\mu t^{-2}\opdiv (\nabla w - \thetav) & =  g  & \mbox{in } \Omega, \label{eq:RMmixed3}
\end{align}
Now, we obtain a variational formulation for finding $\thetav \in \spaceTheta = \Hcurlzero$, $\Mt \in \spaceM = \Hdivdiv$ and $w \in \spaceW = H^1_0$ in the same manner as in Section~\ref{sec:TDNNS},
\begin{align}
    \int_\Omega \At \Mt : \taut\,dx + \langle \opdivv \taut, \thetav\rangle & =  0
        && \forall \taut \in \spaceM, \label{eq:spp_hdivdiv_1} \\
    \langle \opdivv \Mt, \etav\rangle
        - \mu t^{-2} \int_\Omega(\nabla w - \thetav)\cdot(\nabla v - \etav)\,dx
        &= -\int_\Omega gv\,dx && \forall \etav \in \spaceTheta, v \in W. \label{eq:spp_hdivdiv_2}
\end{align}
%
\red{
We shortly comment on the changes in the variational formulation \eqref{eq:spp_hdivdiv_1} - \eqref{eq:spp_hdivdiv_2} which are necessary for the incorporation of different types of boundary conditions. Essential boundary conditions, which have to be enforced by the finite element space, are the deflection $w$, the tangential component of the rotation $\thetav_\tang$ and the normal component of the bending moment $m_{nn}$. Note that these expressions are also degrees of freedom of the corresponding finite element spaces. The corresponding natural conditions, in the same order, are the shear $\mu t^{-2}(\partial_n w - \theta_n)$, the tangential component of the bending moment $\Mv_{nt}$, and the normal component of the rotation $\theta_n$. Natural homogeneous conditions are satisfied whenever the corresponding essential condition is dropped, inhomogeneous conditions result in additional surface integrals on the right hand side, see Table~\ref{tab:bc}.

\begin{table} 
\begin{center}
\begin{tabular}{|l|l|l|}
\hline
essential bc & natural bc & surface term \\\hline
$w$ & $\mu t^{-2} (\partial_n w - \theta_n) = g_0$ & $\int_\Gamma g_0\, v\, ds$ \\
$\thetav_\tang$ & $\Mv_{\nv\tang} = \gv_1$  & $\int_\Gamma \gv_1 \cdot \etav_\tang\, ds$\\
$m_{nn}$ & $\theta_n = g_2$ & $\int_\Gamma g_2\, \tau_{nn}\, ds$ \\\hline
\end{tabular}
\end{center}
\caption{Different types of essential and corresponding natural boundary conditions and the incorporation of the latter into the right hand side of \eqref{eq:spp_hdivdiv_1} - \eqref{eq:spp_hdivdiv_2}.}
\label{tab:bc}
\end{table}

}

The analysis of system \eqref{eq:spp_hdivdiv_1}, \eqref{eq:spp_hdivdiv_2} is subject of Section~\ref{sec:analysis}, while a finite element method is constructed and analyzed in Section~\ref{sec:fem}.

\section{Analysis of the TDNNS Reissner-Mindlin formulation} \label{sec:analysis}

In the current section we show existence and uniqueness of the solution to the Reissner-Mindlin problem in the TDNNS setting. We show stability for decreasing thickness $t \to 0^+$ as well as the limit case $t=0$.

To this end, a further unknown $\gammav = -\mu  t^{-2} (\nabla w - \thetav)$ related to the shear stresses is introduced, see e.g. \cite[Chapter~10.4]{BoffiBrezziFortin:13}. For positive thickness $t$, we choose the corresponding space $\spaceGamma = \LtwotwoOmega$. Problem \eqref{eq:spp_hdivdiv_1}, \eqref{eq:spp_hdivdiv_2} transforms to
\begin{align}
    \int_\Omega \At \Mt : \taut\,dx + \langle \opdivv \taut, \thetav\rangle & =  0 &&
            \forall \taut \in \spaceM, \label{eq:spp_sgtw_1}\\
    \langle \opdivv \Mt, \etav\rangle + \int_\Omega \gammav \cdot(\nabla v - \etav)\,dx
                &= -\int_\Omega gv\,dx &&
            \forall \etav \in \spaceTheta, v \in W, \label{eq:spp_sgtw_2}\\
    \int_\Omega (\nabla w - \thetav) \cdot \deltav\, dx + \mu^{-1} t^2 \int_\Omega \gammav \cdot
            \deltav\, dx &= 0 &&
            \forall \deltav \in \spaceGamma. \label{eq:spp_sgtw_3}
\end{align}
We observe that for system \eqref{eq:spp_sgtw_1} - \eqref{eq:spp_sgtw_3} also the limit case of $t=0$ is well-defined, where the term $\mu^{-1} t^2 \int_\Omega \gammav \cdot
            \deltav\, dx$ vanishes. We reorder terms to obtain a mixed problem in the spirit of \cite{BoffiBrezziFortin:13}. We introduce bilinear forms $\bilA_t: (\spaceM\times \spaceGamma)\times(\spaceM\times \spaceGamma) \to \mathbb R$ depending on $t$ and $\bilB: (\spaceM\times\spaceGamma)\times(\spaceTheta\times\spaceW) \to \mathbb R$ independent of $t$
\begin{eqnarray}
    \bilA_t(\Mt, \gammav; \taut, \deltav) &=&   \int_\Omega \At \Mt : \taut\, dx + \mu^{-1} t^2 \int_\Omega \gammav \cdot \deltav\, dx ,
    \label{eq:def_AA} \\
    \bilB(\Mt, \gammav; \thetav, w) &=& \langle \opdivv \Mt, \thetav\rangle + \int_\Omega \gammav \cdot(\nabla w - \thetav)\,dx.
    \label{eq:def_BB}
\end{eqnarray}
From \eqref{eq:spp_sgtw_1}, \eqref{eq:spp_sgtw_2}, \eqref{eq:spp_sgtw_3} we obtain a saddle point problem of finding $\Mt \in \spaceM$, $\gammav \in \spaceGamma$, $\thetav \in \spaceTheta$ and $w \in \spaceW$ such that
\begin{align}
    \bilA_t(\Mt, \gammav; \taut, \deltav) + \bilB(\taut, \deltav; \thetav, w) & =  0 && \forall \taut \in \spaceM, \deltav \in \spaceGamma \label{eq:spp_ab_1} \\
    \bilB(\Mt, \gammav; \etav, v) & =  -\int_\Omega gv\,dx && \forall \etav \in \spaceTheta, v \in \spaceW. \label{eq:spp_ab_2}
\end{align}

In the limiting case of an infinitely thin (Kirchhoff) plate with $t=0$, it is well known (see e.g. \cite[Proposition~10.4.3]{BoffiBrezziFortin:13}) that for $t\to0$ the shear $\gammav$ stays bounded in $\spaceGamma_0 := \Hmonediv$. We will see that $\spaceGamma_0$ is the natural space for the analysis of the case $t=0$. Note that $\bilA_0(\cdot,\cdot)$ is well-defined on $\spaceM \times \spaceGamma_0$, while for $t>0$ $\bilA_t(\cdot,\cdot)$, cannot be evaluated on the whole space  $\spaceM \times \spaceGamma_0$. The bilinear form $\bilB(\cdot,\cdot)$ is also well-defined in the limiting case, as we shall show below.

For the stability analysis of \eqref{eq:spp_ab_1}, \eqref{eq:spp_ab_2} by Brezzi's theory \cite[Theorem~4.2.3]{BoffiBrezziFortin:13} a characterization of the kernel space $\opker(\opB)$ is needed, which is provided in the following lemma.


\begin{lemma} \label{lemma:kerB}
Define the kernel space
\begin{equation}
    \opker(\opB) := \{(\Mt,\gammav) \in \spaceM \times \spaceGamma:\ \bilB(\Mt,\gammav;\etav,v) = 0 \ \mbox{for all}\ (\etav,v)\in \spaceTheta \times \spaceW\}. \label{eq:defKerB}
\end{equation}
Then any $(\Mt,\gammav) \in \opker(\opB)$ satisfies
\begin{equation} \label{eq:44}
\langle \opdiv \Mt, \nabla v\rangle = 0 \qquad \forall v \in H^1_0
\end{equation}
and
\begin{equation} \label{eq:45}
\langle \opdiv \Mt - \gammav, \etav\rangle = 0 \qquad \forall \etav \in \Hv_0(\opcurl).
\end{equation}
These equalities also hold when $\spaceGamma = \Lv^2$ is replaced by $\spaceGamma_0 = \Hv^{-1}(\opdiv)$. 
\end{lemma}
\begin{proof}
The proof follows directly, setting either $\etav = \nabla v$ or $v=0$ in \eqref{eq:defKerB}. Note that all duality products are well-defined due to the choice of spaces.
\end{proof}

%



\subsection{The limiting case $t=0$} \label{sec:analysis_zerot}

We prove existence and uniqueness of the solution for the limiting case of an infinitely thin (Kirchhoff) plate with $t=0$. 
\red{As mentioned in the introduction, this case is closely related to the Hellan-Herrmann-Johnson formulation, when setting $\thetav = \nabla w$ and eliminating thereby the unknowns $\thetav$ and $\gammav$. However, we will present an analysis of the full mixed system, as it will help understand the case of small thickness $t > 0$.
}

As already mentioned, in this case the natural choice for the shear space is $\spaceGamma_0 := \Hmonediv$. We show boundedness and stability estimates for the bilinear forms, which implies existence, uniqueness and stability of the solution \cite[Theorem~4.2.3]{BoffiBrezziFortin:13}. We use the following natural norms
\begin{eqnarray}
    \|\Mt, \gammav\|_{\spaceM \times \spaceGamma_0}^2 & := &
        \|\Mt\|_{\Hdivdiv}^2 + \|\gammav\|_{\Hv^{-1}(\opdiv)}^2,\\
    \|\thetav, w\|_{\spaceTheta \times \spaceW} &:=&
        \|\thetav\|_{\Hcurl}^2 + \|w\|_{H^1}^2.
\end{eqnarray}

\begin{lemma} \label{lemma:stabA0}
The bilinear form $\bilA_0: (\spaceM\times\spaceGamma_0)\times(\spaceM\times\spaceGamma_0)$ is bounded, for all $\Mt,\taut \in \spaceM$ and $\gammav,\deltav \in \spaceGamma_0$
\begin{equation}
\bilA_0(\Mt,\gammav;\taut,\deltav) \leq \bar\alpha_0 \|\Mt,\gammav\|_{\spaceM\times\spaceGamma_0}\|\taut,\deltav\|_{\spaceM\times\spaceGamma_0}.
\label{eq:Abounded0}
\end{equation}
Moreover, it is coercive on $\opker(\opB)$, for all $(\Mt,\gammav) \in \opker(\opB)$
\begin{equation}
\bilA_0(\Mt,\gammav;\Mt,\gammav) \geq \underline{\alpha}_0 \|\Mt,\gammav\|_{\spaceM\times\spaceGamma_0}^2.
\label{eq:Acoercive0}
\end{equation}
\end{lemma}
\begin{proof}
Boundedness of $\bilA_0$ is straightforward, since $\spaceM \subset \Lt^2$, and
\begin{equation}
	\bilA_0(\Mt,\gammav;\taut,\deltav) \leq \lambda_{max}(\At)\|\Mt\|_{\Lt^2}^2 \leq \lambda_{max}(\At)\|\Mt,\gammav\|_{\spaceM\times\spaceGamma_0}.
\end{equation}

We proceed to showing coercivity of $\bilA_0$ on $\opker(\opB)$.
We use equations \eqref{eq:44} and \eqref{eq:45} of Lemma~\ref{lemma:kerB} to bound $\gammav$ by $\Mt$ in their respective norms. Note that for $\vv \in \Hv^1_0$ and $w \in H^1_0$ we have $\vv \in \Hv_0(\opcurl)$ and $\nabla w \in \Hv_0(\opcurl)$.
\begin{eqnarray}
    \|\gammav\|_{\Hmonediv}^2 & \stackrel{\eqref{eq:8}}{=} & \sup_{\vv \in \Hv^1_0} \frac{\langle \gammav,\vv\rangle^2}{\|\nabla \vv\|_{\Lt^2}^2} + \sup_{w \in H^1_0} \frac{\langle \gammav, \nabla w\rangle^2}{\|\nabla w\|_{\Lv^2}^2} \\
    & \stackrel{\eqref{eq:45}}{=} & \sup_{\vv \in \Hv^1_0} \frac{\langle \opdivv\Mt,\vv\rangle^2}{\|\nabla \vv\|_{\Lt^2}^2} + \sup_{w \in H^1_0} \frac{\langle \opdivv \Mt, \nabla w\rangle^2}{\|\nabla w\|_{\Lv^2}^2}\\
    & \stackrel{\eqref{eq:def_dualityproduct},\eqref{eq:44}}{=} & \sup_{\vv \in \Hv^1_0} \frac{\left(-\int_\Omega \Mt:\epst(\vv)\,dx+0\right)^2}{\|\nabla \vv\|_{\Lt^2}^2} + 0\\
		& \leq & \|\Mt\|_{\Lt^2}^2.
\end{eqnarray}
Now coercivity of $\bilA_0$ is \red{ensured by the minimal eigenvalue $\lambda_{min}(\At)$ of the compliance tensor $\At$,}
\begin{eqnarray}
    \bilA_0(\Mt,\gammav; \Mt,\gammav) &=& \int_{\Omega}(\At \Mt):\Mt\,dx \\
     &\geq& \lambda_{min}(\At)\|\Mt\|_{\Lt^2}^2 + \underbrace{\sup_{w \in H^1_0} \frac{\langle \opdivv \Mt, \nabla w\rangle^2}{\|\nabla w\|_{\Lv^2}^2}}_{=0\ \mathrm{by}\ \eqref{eq:45}} \\
     & \geq & \frac{1}{2}\lambda_{min}(\At)\|\Mt, \gammav\|_{\spaceM \times \spaceGamma_0}^2.
\end{eqnarray}

\end{proof}

\begin{lemma} \label{lemma:stabB0}
The bilinear form $\bilB: (\spaceM\times\spaceGamma_0)\times(\spaceTheta\times\spaceW)$ is bounded, for all $\taut \in \spaceM$, $\deltav \in \spaceGamma_0$, $\etav \in \spaceTheta$ and $v \in \spaceW$
\begin{equation}
\bilB(\taut,\deltav; \etav,v) \leq \bar\beta_0 \|\taut,\deltav\|_{\spaceM\times\spaceGamma_0}\|\etav,v\|_{\spaceTheta\times\spaceW}.
\label{eq:Bbounded0}
\end{equation}
Moreover, it is satisfies an inf-sup condition, for all $\thetav \in \spaceTheta$, $w\in \spaceW$ there exist $\Mt \in \spaceM$, $\gammav \in \spaceGamma$ such that
\begin{equation}
\bilB(\Mt,\gammav;\thetav,w) \geq \underline{\beta}_0 \|\Mt,\gammav\|_{\spaceM\times\spaceGamma_0} \|\thetav,w\|_{\spaceTheta\times\spaceW}.
\label{eq:Binfsup0}
\end{equation}
\end{lemma}
\begin{proof}
Boundedness follows directly by the choice of spaces. We proceed to show the inf-sup condition.
Let $\thetav \in \spaceTheta, w \in \spaceW$ be fixed. 
%
From \cite{PechsteinSchoeberl:NN} we know that $\langle \opdivv \Mt, \thetav\rangle$ is inf-sup stable on $\Hdivdiv \times \Hcurlzero$, i.e.\ there exists some $\tilde \Mt \in \Hdivdiv$ such that
\begin{equation}
    \langle \opdivv \tilde\Mt, \thetav \rangle \geq c_1 \|\tilde\Mt\|_{\Hdivdiv} \|\thetav\|_{\Hcurl} \label{eq:infsup_mtilde}
\end{equation}
and we have shown
\begin{equation}
    \underline{c}\|\thetav\|_{\Hcurl} \leq \|\tilde\Mt\|_{\Hdivdiv} \leq \overline{c} \|\thetav\|_{\Hcurl}, \label{eq:bound_mtilde}
\end{equation}
with $c_1 > 0$ and $0 < \underline{c} \leq 1 \leq \overline{c}$.
Moreover, since $\Hmonediv = \spaceGamma_0$ is the dual space of $\Hcurlzero = \spaceTheta$, and since $\nabla \spaceW \subset \spaceTheta$, by the Riesz Isomorphism there exists some $\tilde\gammav \in \spaceGamma_0$ such that
\begin{equation}
    \|\tilde\gammav\|_{\Hmonediv} = \|\nabla w - \thetav\|_{\Hcurl}
    \label{eq:boundgamma2}
\end{equation}
and
\begin{equation}
    \langle\tilde\gammav, \nabla w - \thetav\rangle = \|\nabla w - \thetav\|_{\Hcurl}^2.
    \label{eq:boundgamma1}
\end{equation}

In the remainder of the proof we verify that the pair $(\Mt,\gammav) := (\frac{2}{c_1\underline{c}}\tilde\Mt, \tilde \gammav)$ satisfies the inf-sup condition \eqref{eq:Binfsup0} with stability constant $\underline{\beta}_0 = \frac{c_1\underline{c}}{\sqrt{2}\overline{c}}$.
First, we observe
\begin{eqnarray}
    \bilB(\Mt, \gammav; \thetav, w) & = & \langle \opdivv \Mt, \thetav\rangle + \langle \gammav, \nabla w - \thetav\rangle \label{eq:bilB} \\
    & \geq & c_1 \|\Mt\|_{\Hdivdiv}\|\thetav\|_{\Hcurl} + \|\gammav\|_{\Hmonediv}\|\nabla w - \thetav\|_{\Hcurl} \\
    & \geq & 2 \|\thetav\|_{\Hcurl}^2 + \|\nabla w - \thetav\|_{\Hcurl}^2.
\end{eqnarray}
Using the triangle inequality in the form that $\|\thetav\|_{\Hcurl}^2 + \|\nabla w - \thetav\|_{\Hcurl}^2 \geq \frac{1}{2} \|\nabla w\|_{\Hcurl}^2$, and the fact that $\|\nabla w\|_{\Hcurl} = \|\nabla w\|_{\Lv^2}$, we obtain
\begin{eqnarray}
    \lefteqn{\bilB(\Mt, \gammav; \thetav, w)} \\
    & \geq &(2\|\thetav\|_{\Hcurl}^2 + \|\nabla w - \thetav\|_{\Hcurl}^2)^{1/2} (\|\thetav\|_{\Hcurl}^2 + \frac{1}{2}\|\nabla w\|_{\Lv^2}^2)^{1/2} .
\end{eqnarray}
%
%
Since $2\|\thetav\|_{\Hcurl} \geq \frac{2}{\overline{c}}\|\tilde \Mt\|_{\Hdivdiv} = \frac{c_1\underline{c}}{\overline{c}} \|\Mt\|_{\Hdivdiv}$ due to equation \eqref{eq:bound_mtilde} and the definition of $\Mt$, and $\|\gammav\|_{H^{-1}(\opdiv)} = \|\nabla w - \thetav\|_{H(\opcurl)}$ due to equation \eqref{eq:boundgamma2}, we conclude
\begin{eqnarray}
\lefteqn{\bilB(\Mt, \gammav; \thetav, w)}\\
 &\geq &\frac{c_1\underline{c}}{\overline{c}}\Big(  \|\Mt\|_{\Hdivdiv}^2 + \|\gammav\|_{\Hmonediv}^2 \Big)^{1/2} \big(\|\thetav\|_{\Hcurl}^2 + \frac{1}{2}\|\nabla w\|_{\Lv^2}^2\big)^{1/2} \\
		&=& \frac{c_1\underline{c}}{\sqrt{2}\overline{c}} \|\Mt,\gammav\|_{\spaceM\times\spaceGamma_0} \|\thetav,w\|_{\spaceTheta\times\spaceW}.
\end{eqnarray}
\end{proof}

\begin{theorem} \label{theo:uniquet0}
For $t=0$, problem \eqref{eq:spp_ab_1}, \eqref{eq:spp_ab_2} has a unique solution $\Mt \in \spaceM$, $w \in \spaceW$, $\thetav \in \spaceTheta$ and $\gammav \in \spaceGamma_0 = \Hmonediv$. The solution is bounded by
\begin{equation}
    \|\Mt\|_{\Hdivdiv} + \|\gammav\|_{\Hmonediv} + \|\thetav\|_{\Hcurl} +
 \|w\|_{H^1} \leq c \|g\|_{H^{-1}} \label{eq:bound_t0}
\end{equation}
with $c$ a generic constant.
\end{theorem}
%

\subsection{The case of positive thickness $t>0$} \label{sec:analysis_post}

In this section, we prove existence and uniqueness of a solution to the Reissner-Mindlin problem \eqref{eq:spp_sgtw_1} - \eqref{eq:spp_sgtw_3} in the case of positive thickness $t>0$. To this end, a different set of norms is introduced, that includes the thickness $t$.
\begin{eqnarray}
    \|\Mt, \gammav\|_{\spaceM \times \spaceGamma,t}^2 & := &
        \|\Mt\|_{\Hdivdiv}^2 + t\|\gammav\|_{\Lv^2}^2,\\
    \|\thetav, w\|_{\spaceTheta \times \spaceW,t}^2 &:=&
        \|\thetav\|_{\Hcurl}^2 + \|w\|_{H^1}^2 + t^{-2}\|\nabla w - \thetav\|_{\Lv^2}^2.
        \label{eq:def_normthetaw_t}
\end{eqnarray}

\begin{lemma} \label{lem:normequiv}
A norm equivalent to $\|\thetav, w\|_{\spaceTheta \times \spaceW}$
can be defined omitting the term $\|w\|_{H^1}$ in
\eqref{eq:def_normthetaw_t}, where the non-trivial bound is
characterized by the Friedrichs constant $c_F$,
\begin{equation}
    \|\thetav, w\|_{\spaceTheta \times \spaceW,t}^2
    \leq (1+2(1+c_F^2)) \left(\|\thetav\|_{\Hcurl}^2 + t^{-2}\|\nabla w -
    \thetav\|_{\Lv^2}^2\right).
\end{equation}
\end{lemma}
\proof{ The statement of the lemma is clear from the following
consideration, which uses Friedrichs' inequality, the triangle
inequality and the fact that $t<1$.
\begin{eqnarray}
    \frac{1}{1+c_F^2}\|w\|_{H^1}^2 &\leq& \|\nabla w\|_{\Lv^2}^2 \leq 2(\|\thetav\|_{\Lv^2}^2 + \|\nabla w -\thetav\|_{\Lv^2}^2)\\
    &\leq& 2(\|\thetav\|_{\Hcurl}^2 + t^{-2}\|\nabla w-\thetav\|_{\Lv^2}^2).
\end{eqnarray}

}

The next two lemmas provide stability estimates for the bilinear forms in the $t$-dependent norms $\|\Mt,\gammav\|_{\spaceM\times\spaceGamma,t}$ and $\|\thetav,w\|_{\spaceTheta\times\spaceW,t}$. These estimates are used to ensure existence and uniqueness of a solution and to obtain a stability estimate not deteriorating with $t \to 0$. The proof of Lemma~\ref{lemma:stabAt} is straightforward in the manner of the proof of Lemma~\ref{lemma:stabA0}:

\begin{lemma} \label{lemma:stabAt}
For $t>0$, the bilinear form $\bilA_t: (\spaceM\times\spaceGamma)\times(\spaceM\times\spaceGamma)$ is bounded, for all $\Mt,\taut \in \spaceM$ and $\gammav,\deltav \in \spaceGamma$
\begin{equation}
\bilA_t(\Mt,\gammav;\taut,\deltav) \leq \bar\alpha \|\Mt,\gammav\|_{\spaceM\times\spaceGamma,t}\|\taut,\deltav\|_{\spaceM\times\spaceGamma,t}.
\label{eq:Aboundedt}
\end{equation}
It is coercive on $\opker(\opB)$, for all $(\Mt,\gammav) \in \opker(\opB)$
\begin{equation}
\bilA_t(\Mt,\gammav;\Mt,\gammav) \geq \underline{\alpha} \|\Mt,\gammav\|_{\spaceM\times\spaceGamma,t}^2.
\label{eq:Acoercivet}
\end{equation}
The constants $\bar \alpha$, $\underline{\alpha}$ are independent of the thickness $t$.
\end{lemma}

\begin{lemma} \label{lemma:stabBt}
The bilinear form $\bilB: (\spaceM\times\spaceGamma)\times(\spaceTheta\times\spaceW)$ is continuous with respect to the $t$-dependent norms, for all $\taut \in \spaceM$, $\deltav \in \spaceGamma$, $\etav \in \spaceTheta$ and $v \in \spaceW$
\begin{equation}
\bilB(\taut,\deltav; \etav,v) \leq \bar\beta \|\taut,\deltav\|_{\spaceM\times\spaceGamma,t}\|\etav,v\|_{\spaceTheta\times\spaceW,t}.
\label{eq:Bboundedt}
\end{equation}
Moreover, it is satisfies an inf-sup condition, for all $\thetav \in \spaceTheta$, $w\in \spaceW$ there exist $\Mt \in \spaceM$, $\gammav \in \spaceGamma$ such that
\begin{equation}
\bilB(\Mt,\gammav;\thetav,w) \geq \underline{\beta} \|\Mt,\gammav\|_{\spaceM\times\spaceGamma,t} \|\thetav,w\|_{\spaceTheta\times\spaceW,t}.
\label{eq:Binfsupt}
\end{equation}
The constants $\bar \beta$, $\underline{\beta}$ are independent of the thickness $t$.
\end{lemma}
\begin{proof}
Obviously, the bilinear form is bounded, as the divergence term is bounded in $\Ht(\opdiv\opdivv) \times \Hv(\opcurl)$, and the integral is bounded by the respective scaled $\Lv^2$-norms $t\|\deltav\|_{\Lv^2}$ and $t^{-1}\|\nabla w - \theta\|_{\Lv^2}$.

To prove the inf-sup condition, assume $\thetav \in \spaceTheta$ and $w \in \spaceW$ are given.
Similar to the proof of Lemma~\ref{lemma:stabB0}, we use the theory provided in \cite{PechsteinSchoeberl:NN}. We choose $\Mt=\tilde \Mt \in \spaceM = \Hdivdiv$ from equations \eqref{eq:infsup_mtilde} and \eqref{eq:bound_mtilde}.
Additionally, we choose $\gammav = t^{-2}(\nabla w - \thetav) $. This is possible since $\nabla w - \thetav \in \Hcurl \subset \mathbf{L}^2$. Then we have
\begin{equation}
    \int_{\Omega} \gammav \cdot( \nabla w - \thetav)\,dx  = t^{-2}\|\nabla w - \thetav\|_{\Lv^2}^2. \label{eq:gamma}
\end{equation}
Combining the definition of $\bilB(\cdot,\cdot)$ \eqref{eq:def_BB}, equation \eqref{eq:gamma} and the bounds from \eqref{eq:infsup_mtilde} and \eqref{eq:bound_mtilde} we obtain
\begin{eqnarray}
    \bilB(\Mt,\gammav;\thetav,w) &=& \langle \opdivv \Mt, \thetav \rangle + \int_{\Omega}\gammav \cdot ( \nabla w - \thetav)\,dx  \label{eq:stability_t_1}\\
    & \geq & c_1\underline{c}\|\thetav\|_{\Hcurl}^2 + t^{-2}\|\nabla w - \thetav\|_{\Lv^2}^2.
\end{eqnarray}
Basic algebra for real numbers, the definition  of the shear $\gammav$ and the upper bound in \eqref{eq:bound_mtilde} lead to
\begin{align}
    &\bilB(\Mt,\gammav;\thetav,w)\\
    & \geq  c_1 \underline{c} \big(\|\thetav\|_{\Hcurl}^2 + t^{-2}\|\nabla w - \thetav\|_{\Lv^2}^2\big) \\
    & \geq  \frac{c_1 \underline{c}}{\overline{c}} \left( \|\Mt\|_{\Hdivdiv}^2+t^2\|\gammav\|_{\Lv^2}^2\right)^{1/2}
    \big(\|\thetav\|_{\Hcurl}^2 + t^{-2}\|\nabla w - \thetav\|_{\Lv^2}^2\big)^{1/2}.
\end{align}
We see inf-sup stability for $\bilB(\cdot,\cdot)$ using the bound from Lemma~\ref{lem:normequiv}
\begin{eqnarray}
    \bilB(\Mt,\gammav;\thetav,w) & \geq & \frac{c_1 \underline{c}}{\overline{c}\sqrt{1+2(1+c_F^2)}}  \|\Mt,\gammav\|_{\spaceM\times \spaceGamma,t} \|\thetav,w\|_{\spaceTheta\times \spaceW,t}. \label{eq:stability_t_n}
\end{eqnarray}

\end{proof}

\begin{theorem} \label{theo:stability_t}
For $t>0$, problem \eqref{eq:spp_ab_1}, \eqref{eq:spp_ab_2} has a unique solution $\Mt \in \spaceM$, $w \in \spaceW$, $\thetav \in \spaceTheta$ and $\gammav \in \spaceGamma$, which is bounded as below
\begin{equation}
    \|\Mt\|_{\Hdivdiv} + \|\gammav\|_{\Hv^{-1}(\opdiv)} + t\|\gammav\|_{\Lv^2} + \|\thetav\|_{\Hcurl} +
 \|w\|_{H^1} \leq c \|g\|_{H^{-1}} \label{eq:bound_t}
\end{equation}

where $c$ is a generic constant independent of $t$.
\end{theorem}
\begin{proof}
Again, we use the statement of \cite[Theorem~4.2.3]{BoffiBrezziFortin:13}, where coercivity (Lemma~\ref{lemma:stabAt}) and inf-sup stability (Lemma~\ref{lemma:stabBt}) ensure the existence and stability of a unique solution. Note that, for the solution $\Mt \in \spaceM$, $w \in \spaceW$, $\thetav \in \spaceTheta$ and $\gammav \in \spaceGamma$, there holds $\gammav = \mu t^{-2} (\nabla w - \thetav)$. Thus $t\|\gammav\|_{\Lv^2} = t^{-1}\|\nabla w - \thetav\|_{\Lv^2}$, which ensures the estimate
\begin{equation}
    \|\Mt\|_{\Hdivdiv} + t\|\gammav\|_{\Lv^2} + \|\thetav\|_{\Hcurl} +
 \|w\|_{H^1} \leq c \|g\|_{H^{-1}}.
\end{equation}

We add the bound on $\|\gammav\|_{\Hv^{-1}(\opdiv)}$: since $\Hv^{-1}(\opdiv)$ is dual to $\spaceTheta = \Hv(\opcurl)$ and the solution $\gammav$ satisfies the variational equation \eqref{eq:spp_sgtw_2} with $v = 0$ we have
\begin{equation}
	\|\gammav\|_{\Hv^{-1}(\opdiv)} = \sup_{\etav \in \spaceTheta} \frac{\langle \gammav, \etav\rangle}{\|\etav\|_{\spaceTheta}} = \sup_{\etav \in \spaceTheta} \frac{\langle \opdivv \Mt, \etav \rangle}{\|\etav\|_{\spaceTheta}} \leq \|\Mt\|_{\spaceM}.
\end{equation}
Thus, the statement of the theorem is shown.

\end{proof}


\section{Finite Elements} \label{sec:fem}

Throughout this section, let $\Omega \subset \mathbb R^2$ be a polygonal Lipschitz domain, and let $(\Triang_h)$ be a family of decompositions into triangular elements $T$. We assume that the family $(\Triang_h)$ is regular, shape-regular and quasi-uniform with mesh size $h$ (see e.g. \cite{BrennerScott:02}). Moreover, $(\Facets_h)$ shall denote the set of edges in the mesh.

We propose a finite element method without using the additional unknown shear $\gammav$, having the form of \eqref{eq:spp_hdivdiv_1}-\eqref{eq:spp_hdivdiv_2}. We use the TDNNS finite element spaces for the bending moments and rotations, and a fully continuous Lagrange space for the deflection. For integer $k \geq 1$, $\Mt$ is approximated in the normal-normal continuous space of order $k$ introduced in \cite{PechsteinSchoeberl:11}, while $\thetav$ is discretized by order $k$ N\'ed\'elec elements of the second kind \cite{Nedelec:86}. The deflection elements are order $k+1$ continuous elements. In detail, we choose
\begin{eqnarray}
    \spaceM_h  & := &
        \{\Mt_h \in \LtwosymOmega: \Mt_h|_T \in P^k, \M_{h,nn} \mbox{ continuous}\},\\
    \spaceTheta_h & := &
        \{\thetav_h \in \LtwotwoOmega: \thetav_h|_T \in P^k, \thetav_{h,\tang} \mbox{ continuous}, \thetav_{h,\tang} = 0 \mbox{ on } \partial \Omega\},\\
    \spaceW_h & := &
        \{w_h \in H^1_0(\Omega): w_h|_T \in P^{k+1}\}.
\end{eqnarray}
The discrete system used in implementations reads
\begin{align}
    \int_\Omega \At \Mt_h : \taut_h\,dx + \langle \opdivv \taut_h, \thetav_h\rangle & =  0
        \qquad \forall \taut_h \in \spaceM_h, \label{eq:spp_hdivdiv_1h} \\
    \langle \opdivv \Mt_h, \etav_h\rangle
        - \mu t^{-2} \int_\Omega(\nabla w_h - \thetav_h)\cdot(\nabla v_h - \etav_h)\,dx
        &= -\int_\Omega g\,v_h\,dx  \label{eq:spp_hdivdiv_2h}\\
        & \qquad\forall \etav_h \in \spaceTheta_h, v_h \in \spaceW_h.  \nonumber
\end{align}

Note that the finite element space $\spaceM_h$ is (slightly) non-conforming, $\spaceM_h \not \subset \spaceM = \Ht(\opdiv \opdivv)$. 
\red{This is due to lacking continuity of $\Mv_{\nv\tang}$ at the corner points (in the interior) of each element, see \cite[page 13]{PechsteinSchoeberl:NN} for a detailed discussion. However, the duality product $\langle \opdivv \taut_h, \thetav_h\rangle$ can now  be understood as the $\spaceTheta_h^* \times \spaceTheta_h$ duality product, and can be evaluated by the relations \eqref{eq:def_div1} - \eqref{eq:def_div2}.}
Moreover, the norm $\|\cdot\|_{\spaceM}$ from \eqref{eq:norm_hdivdiv} is not well-defined for $\Mt_h \in \spaceM_h$. In \cite{PechsteinSchoeberl:NN} we provided a discrete norm and a corresponding stability analysis for the TDNNS continuum mechanics elements. We will use this discrete norm in the current paper, defining
\begin{equation}
	\|\Mt\|_{\spaceM_h}^2 := \|\Mt\|_{\Lt^2}^2 + \sum_{F \in \mathcal E} h_F\|m_{nn}\|_{L^2(F)}^2 + \sup_{w_h \in \spaceW_h} \frac{\langle \opdivv \Mt, \nabla w_h\rangle^2}{\|\nabla w_h\|_{\Lv^2(\Omega)}^2}, \label{eq:norm_hdivdiv_h}
\end{equation}
and the parameter-dependent norm
\begin{eqnarray}
    \|\Mt, \gammav\|_{\spaceM_h \times \spaceGamma_h,t}^2 & := &
        \|\Mt\|_{\spaceM_h}^2  + t^2\|\gammav\|_{L^2}^2.
\end{eqnarray}

For finite element tensors $\Mt_h \in \spaceM_h$, the edge $L^2$ terms in the norm above can also be omitted, as they are bounded by the domain $\Lt^2$ norm.
The divergence operator $\opdivv: \spaceM_h \to \spaceTheta_h^*$ is bounded and LBB-stable, see \cite{PechsteinSchoeberl:NN}:
\begin{theorem} \label{theo:div_stable_h}
There exist positive constants $\beta_1, \beta_2 > 0$  such that for any $\Mt_h \in \spaceM_h$ and $\thetav_h \in \spaceTheta_h$
\begin{equation}
	\langle \opdivv \Mt_h, \thetav_h\rangle \leq \beta_1 \|\Mt_h\|_{\spaceM_h} \|\theta_h\|_{\Hv(\opcurl)},
\end{equation}
and
\begin{equation}
	\inf_{\thetav_h \in \spaceTheta_h} \sup_{\Mt_h \in \spaceM_h} \frac{\langle \opdivv \Mt_h, \thetav_h\rangle}{\|\Mt_h\|_{\spaceM_h} \|\thetav_h\|_{\Hv(\opcurl)}} \geq \beta_2.
\end{equation}
\end{theorem}

\subsection{Discrete stability}

For the analysis, it is convenient to introduce a finite element discretization for the shear $\gammav$, which leads to a discrete system equivalent to \eqref{eq:spp_hdivdiv_1h}, \eqref{eq:spp_hdivdiv_2h}, but which is of the standard saddle point form \eqref{eq:spp_ab_1}, \eqref{eq:spp_ab_2}. The equivalence of the discrete systems is due to the inclusion $\nabla \spaceW_h \subset \spaceTheta_h$, and to our choice $\spaceGamma_h = \spaceTheta_h$.
Thus, for $w_h \in \spaceW_h$, $\thetav_h \in \spaceTheta_h$ and $\gammav_h \in \spaceGamma_h$ the discrete variational equation
\begin{equation}
    \int_{\Omega} (\nabla w_h - \thetav_h)\cdot \deltav_h\,dx = \mu^{-1}t^2 \int_{\Omega} \gammav_h\cdot \deltav_h\,dx
        \qquad \forall \deltav_h \in \spaceGamma_h
				\label{eq:discr_eq_gamma}
\end{equation}
is equivalent to $\gammav_h = \mu t^{-2}(\nabla w_h - \thetav_h)$. This implies that $\gammav_h$ can be eliminated, and the smaller system \eqref{eq:spp_hdivdiv_1h}, \eqref{eq:spp_hdivdiv_2h} may be used in implementations.

The stability analysis is similar to the analysis of the infinite dimensional problem for positive thickness presented in Section~\ref{sec:analysis_post}.

\begin{lemma}
The bilinear form $\bilA_t: (\spaceM_h \times \spaceGamma_h) \times (\spaceM_h \times \spaceGamma_h)$ is coercive on $\opker(\opB_h) := \{(\Mt_h, \gammav_h) \in \spaceM_h \times \spaceGamma_h: \bilB(\Mt_h,\gammav_h;\thetav_h,w_h) = 0 \ \forall \thetav_h \in \spaceTheta_h, w_h \in \spaceW_h\}$. There exists a constant  $\alpha_1 > 0$ independent of $t, h$ such that
\begin{equation}
\bilA_t(\Mt_h,\gammav_h;\Mt_h,\gammav_h) \geq \alpha_1 \|\Mt_h,\gammav_h\|_{\spaceM_h \times \spaceGamma_h,t}^2.
\end{equation}
\end{lemma}
\begin{proof}
For any $(\Mt_h, \gammav_h) \in \opker(\opB_h)$ we have by definition, setting $\thetav_h = \nabla w_h$,
\begin{equation}
	\langle \opdivv \Mt_h, \nabla w_h \rangle = 0.
\end{equation}
Thus, it follows
\begin{align}
	&\bilA_t(\Mt_h,\gammav_h;\thetav_h,w_h) = \int_{\Omega}(\At \Mt_h):\Mt_h\,dx + \mu^{-1} t^2 \int_\Omega \gammav_h : \gammav_h\,dx \\
     &\geq \lambda_{min}(\At)(\|\Mt_h\|_{\Lt^2}^2 + 
				\underbrace{\sup_{w_h \in \spaceW_h} \frac{\langle \opdivv \Mt_h,\nabla w_h\rangle^2}{\|\nabla w_h\|_{\Lv^2(\Omega)}^2}}_{=0}) + \mu^{-1} t^2 \|\gammav_h\|_{\Lv^2(\Omega)}^2\\
     & \geq  \min(\lambda_{min}(\At),\mu^{-1}) \|\Mt_h, \gammav_h\|_{\spaceM_h \times \spaceGamma_h,t}^2.
\end{align}
\end{proof}

\begin{lemma}
The bilinear form $\bilB: (\spaceM_h \times \spaceGamma_h) \times (\spaceTheta_h \times \spaceW_h)$ is bounded and inf-sup stable, for any $\thetav_h \in \spaceTheta_h, w_h \in \spaceW_h$ there exist $\Mt_h \in \spaceM_h, \gammav_h \in \spaceGamma_h$ such that
\begin{equation}
	\bilB(\Mt_h, \gammav_h;\thetav_h,w_h) \geq \beta \|\Mt_h,\gammav_h\|_{\spaceM_h\times\spaceGamma_h,t} \|\thetav_h, w_h\|_{\spaceTheta\times\spaceW,t}.
\end{equation}
\end{lemma}
Boundedness is clear from the discrete boundedness of the divergence operator, see Theorem~\ref{theo:div_stable_h}.

The discrete inf-sup condition is shown in the same manner as in the infinite dimensional case in Theorem~\ref{theo:stability_t}. The arguments shall not be repeated here, but only shortly commented on. 

We use the discrete inf-sup stability of the divergence operator from Theorem~\ref{theo:div_stable_h}.
As in Theorem~\ref{theo:stability_t}, we set $\gammav_h = t^{-2}(\nabla w_h - \thetav_h)$, which is possible since $\nabla \spaceW_h \subset \spaceTheta_h = \spaceGamma_h$. The remainder of the proof involves the same steps as shown in eq.~\eqref{eq:stability_t_1}-\eqref{eq:stability_t_n}, only replacing the infinite-dimensional norms by the discrete ones.

\subsection{A-priori error estimates} \label{sec:interpolation}

To get a-priori error estimates, it is necessary to have interpolation error estimates. 
We use the standard nodal interpolation operator $\IntW$ for $H^1$ and the standard interpolator $\IntTheta$ of the N\'ed\'elec space defined using its degrees of freedom, see e.g.\ \cite{Monk:03} for their definition. The following approximation properties for sufficiently smooth functions are provided there for $1 \leq m \leq k$,
\begin{eqnarray}
    \|\thetav - \IntTheta\thetav\|_{\Hv(\opcurl)}^2 & \leq & c  \sum_{T\in \Triang_h} h^{2m} \|\thetav\|^2_{\Ht^{m+1}(T)} , \label{eq:est_inttheta}\\
    \|w - \IntW w\|_{H^1}^2 & \leq & c \sum_{T\in \Triang_h} h^{2m} \|w\|^2_{H^{m+1}(T)} .
\end{eqnarray}
An important property of the interpolation operators is that they commute with the gradient operator, see e.g.\ \cite[Theorem 5.49]{Monk:03},
\begin{equation}
    \IntTheta \nabla w = \nabla \IntW w.
\label{eq:commutingdiagram}
\end{equation}

For the bending moments $\Mt$ we use the nodal interpolation operator $\IntSigma$, which is provided and analyzed in \cite{PechsteinSchoeberl:NN}. An error estimate in the discrete $\Ht(\opdiv\opdivv)$ norm was found for $\IntSigma$ for $0 \leq l \leq k$
\begin{eqnarray}
    \|\Mt - \IntSigma\Mt\|_{\spaceM_h}^2 & \leq & c  \sum_{T\in \Triang_h} h^{2(l+1)} \|\Mt\|^2_{\Ht^{l+1}(T)} .
\end{eqnarray}
 We shall not provide the degrees of freedom of the stress space in detail here, only note that in two space dimensions there are edge-based degrees of freedom coupling the normal-normal component $m_{nn}$ of $\Mt$, and inner degrees of freedom which can be eliminated by static condensation. Corresponding polynomial basis functions can be found in \cite{PechsteinSchoeberl:11}.

We can now venture to show a convergence result of the proposed finite element method.

\begin{theorem}
Let $\Mt \in \spaceM$, $w \in \spaceW$ and $\thetav \in \spaceTheta$ be the exact solution to the Reissner-Mindlin problem \eqref{eq:spp_hdivdiv_1}, \eqref{eq:spp_hdivdiv_2}, and let $\Mt_h \in \spaceM_h$, $w_h \in \spaceW_h$ and $\thetav_h \in \spaceTheta_h$ be the corresponding finite element solution. Then we have the a-priori error estimate for $1 \leq m \leq k$

\begin{align}
    &\|\thetav-\thetav_h\|_{\Hv(\opcurl)} + \|w-w_h\|_{H^1(\Omega)} + \|\Mt - \Mt_h\|_{\spaceM_h} +t\|\gammav - \gammav_h\|_{\Lv^2(\Omega)} \\
    &\leq c \left(\sum_{T\in \Triang_h} h^{2m} (\|\thetav\|^2_{\Hv^{m+1}(T)} + \|\Mt\|^2_{\Ht^{m}(T)} + t^2\|\gammav\|_{\Hv^m(T)}^2)  \right)^{1/2}. \label{eq:errorbound}
\end{align}

\end{theorem}
\begin{proof}
Since the finite element method is slightly nonconforming, $\spaceM_h \not \subset \spaceM$, see \cite{PechsteinSchoeberl:NN}, we use techniques from Strang's second lemma. We bound the total error \eqref{eq:totalerror} by interpolation error \eqref{eq:interperror} and consistency error \eqref{eq:consistencyerror}. 
\begin{align}
	\|\thetav &	 - \thetav_h\|_{\Hv(\opcurl)} + \|w - w_h\|_{H^1(\Omega)} + 
						\|\Mt - \Mt_h\|_{\spaceM_h}+t\|\gammav - \gammav_h\|_{\Lv^2(\Omega)}
						\label{eq:totalerror}\\
	\leq&   	\left\{\begin{array}{l}
						\|\thetav - \IntTheta\thetav\|_{\Hv(\opcurl)} + \|w - \IntW w\|_{H^1(\Omega)} + \\
						\|\Mt - \IntSigma \Mt\|_{\spaceM_h}+t\|\gammav - \IntTheta\gammav\|_{\Lv^2(\Omega)} 
						\end{array}\right\}
						\label{eq:interperror} + \\
	&  				\left\{\begin{array}{l}
						\|\IntTheta\thetav - \thetav_h\|_{\Hv(\opcurl)} + \|\IntW w - w_h\|_{H^1(\Omega)} + \\
						\|\IntSigma \Mt - \Mt_h\|_{\spaceM_h}+t\|\IntTheta\gammav - \gammav_h\|_{\Lv^2(\Omega)} 
						\end{array}\right\}
						\label{eq:consistencyerror}
\end{align}

Clearly, the interpolation error \eqref{eq:interperror} can be bounded as stated above \eqref{eq:errorbound}. We concentrate on the consistency error.

As stated in \cite[Theorem~5.2.1]{BoffiBrezziFortin:13}, discrete stability ensures
\begin{eqnarray}
\lefteqn{\eqref{eq:consistencyerror}} \\
&\leq & \|\IntTheta\thetav - \thetav_h;\IntW w - w_h\|_{\spaceTheta\times\spaceW,t} + \|\IntSigma \Mt - \Mt_h, \IntTheta \gammav - \gammav_h\|_{\spaceM_h\times \spaceGamma_h,t} \\
& \leq & \sup_{\taut_h \in \spaceM_h\atop \deltav_h \in \spaceGamma_h}
				\frac{ \left\{ \begin{array}{l}
							\bilA_t(\IntSigma\Mt - \Mt_h, \IntTheta\gammav-\gammav_h; \taut_h, \deltav_h) + \\
							\bilB(\taut_h,\deltav_h; \IntTheta\thetav - \thetav_h, \IntW w - w_h)
							\end{array}\right\}
							}{\|\taut_h, \deltav_h\|_{\spaceM_h\times\spaceGamma_h,t}} + \\
& &		\sup_{\etav_h \in \spaceTheta_h \atop v_h \in \spaceW_h}
		\frac{\bilB(\IntSigma \Mt - \Mt_h, \IntTheta \gammav - \gammav_h; \nabla v_h - \etav_h)}{\|\etav_h, v_h\|_{\spaceTheta\times\spaceW,t}}\\
& \leq & \sup_{\taut_h\in \spaceM_h} \frac{\int_\Omega \At (\IntSigma\Mt - \Mt_h):\taut_h\,dx + \langle \opdivv \taut_h, \IntTheta\thetav - \thetav_h\rangle}{\|\taut_h\|_{\spaceM_h}} + \label{eq:term1}\\
& & \sup_{\deltav_h\in \spaceGamma_h} \frac{\int_\Omega (\nabla \IntW w-\nabla w_h - \IntTheta\thetav + \thetav_h + \frac{t^2}{\mu}(\IntTheta\gammav-\gammav_h))\cdot \deltav_h\,dx }{t\|\deltav_h\|_{\Lv^2(\Omega)}} + \label{eq:term2}\\
& & \sup_{\etav_h \in \spaceTheta_h \atop v_h \in \spaceW_h}
		\frac{\langle \opdivv (\IntSigma \Mt - \Mt_h), \etav_h\rangle + \int_{\Omega} (\IntTheta \gammav - \gammav_h)\cdot (\nabla v_h - \etav_h)\, dx}{\|\etav_h, v_h\|_{\spaceTheta\times\spaceW,t}} \label{eq:term3}
\end{eqnarray}

The first term, \eqref{eq:term1}, was treated in \cite{PechsteinSchoeberl:NN}, and is bounded by
\begin{equation}
\eqref{eq:term1} \leq c \left(\sum_{T\in \Triang_h} h^{2m} (\|\thetav\|^2_{\Hv^{m+1}(T)}+\|\Mt\|^2_{\Ht^{m}(T)})  \right)^{1/2}
\end{equation}

For the second term \eqref{eq:term2}, we used the commuting diagram property of the interpolation operators $\IntW$ and $\IntTheta$ \eqref{eq:commutingdiagram} and the linearity of $\IntTheta$,
\begin{eqnarray}
\eqref{eq:term2} &=& \sup_{\deltav_h\in \spaceGamma_h} \frac{\int_\Omega (\IntTheta\nabla  w-\nabla w_h - \IntTheta\thetav + \thetav_h + \frac{t^2}{\mu}(\IntTheta\gammav-\gammav_h))\cdot \deltav_h\,dx }{t\|\deltav_h\|_{\Lv^2(\Omega)}} \\
&=& \sup_{\deltav_h\in \spaceGamma_h} \frac{\int_\Omega \big(\IntTheta(\nabla  w-\thetav+ \frac{t^2}{\mu} \gammav)-(\nabla w_h - \thetav_h - \frac{t^2}{\mu}\gammav_h)\big)\cdot \deltav_h\,dx }{t\|\deltav_h\|_{\Lv^2(\Omega)}}.
\end{eqnarray}
The discrete solution $\gammav_h$ satisfies $\gammav_h = \mu t^{-2} (\thetav_h - \nabla w_h)$ (see \eqref{eq:discr_eq_gamma}). As the solution $\gammav$ satisfies $\gammav = \mu t^{-2} (\thetav - \nabla w)$, we obtain
\begin{equation}
\eqref{eq:term2} = 0.
\end{equation}

We proceed to estimating the last term \eqref{eq:term3}. Since $(\Mt_h, \gammav_h)$ and $(\Mt, \gammav)$ are solutions to the discrete and infinite-dimensional variational equations, and since $(\etav_h, w_h) \in \spaceTheta_h \times \spaceW_h \subset \spaceTheta \times \spaceW$, we have
\begin{equation}
 \langle \opdivv (\Mt - \Mt_h), \etav_h\rangle + \int_{\Omega} (\gammav -\gammav_h)\cdot (\nabla v_h - \etav_h)\, dx = 0.
\end{equation}
The first term above and the divergence of similar differences of $\Mt$ and discrete tensors in $\spaceM_h$ is well-defined in the sense
\begin{equation}
\langle \opdivv (\Mt - \Mt_h), \etav_h\rangle = \langle \opdivv \Mt, \etav_h\rangle_{\spaceTheta^*\times \spaceTheta} - \langle \opdivv \Mt_h, \etav_h\rangle_{\spaceTheta_h^*\times \spaceTheta_h}.
\end{equation}
We may rewrite
\begin{equation}
\eqref{eq:term3} = \sup_{\etav_h \in \spaceTheta_h \atop v_h \in \spaceW_h}
		\frac{\langle \opdivv (\IntSigma \Mt - \Mt), \etav_h\rangle + \int_{\Omega} (\IntTheta \gammav - \gammav)\cdot (\nabla v_h - \etav_h)\, dx}{\|\etav_h, v_h\|_{\spaceTheta\times\spaceW,t}}
\end{equation}
In \cite{PechsteinSchoeberl:NN} we have shown that
\begin{equation}
\langle \opdivv (\IntSigma \Mt - \Mt), \etav_h\rangle \leq c \|\IntSigma \Mt - \Mt\|_{\spaceM_h} \|\etav_h\|_{\Hv(\opcurl)}.
\end{equation}
Thus, we deduce
\begin{align}
	&\eqref{eq:term3}\\
	& \leq  c
	\sup_{\etav_h \in \spaceTheta_h \atop v_h \in \spaceW_h}
		\frac{\|\IntSigma \Mt - \Mt\|_{\spaceM_h} \|\etav_h\|_{\Hv(\opcurl)} +\|\IntTheta \gammav - \gammav\|_{\Lv^2} \|\nabla v_h - \etav_h\|_{\Lv^2}}{\|\etav_h, v_h\|_{\spaceTheta\times\spaceW,t}}\\
		& \leq  c \left(\|\IntSigma \Mt - \Mt\|_{\spaceM_h} + t\|\IntTheta \gammav - \gammav\|_{\Lv^2(\Omega)}\right) \\
		& \leq  c \left(\sum_{T\in \Triang_h} h^{2m} \|\Mt\|^2_{\Ht^{m}(T)} + t^2\|\gammav\|_{\Hv^m(T)}^2  \right)^{1/2}.
\end{align}
Consequently, we have arrived at the desired result.
\end{proof}

\subsection{Hybridization}

To avoid the implementation of normal-normal continuous finite elements and an indefinite system matrix, a hybridization technique in the spirit of \cite[Chapter~7.1]{BoffiBrezziFortin:13} was mentioned in \cite{PechsteinSchoeberl:11} and analyzed in \cite{Sinwel:09}. Here, the normal-normal continuity of the tensor of bending moments is broken and imposed by Lagrangian multipliers defined on element edges. The Lagrangian multipliers resemble the normal component of the rotation $\gamma_n$. As the Lagrangian multipliers are chosen of the same polynomial order as the normal-normal component of the bending moment, the discrete systems are equivalent. However, now the bending moment $\Mt$ is completely local and can be eliminated element-wise (static condensation). The remaining system contains only displacement-based unknowns. It is symmetric positive definite, which makes it easier to be solved by sparse direct solver or standard iterative solvers.

\section{Numerical example} \label{sec:numerics}

\subsection{Clamped square plate}

The first example is taken from \cite{ChinosiLovadinaMarini:06}, where the solution is known analytically. We consider a clamped square plate $\Omega = (0,1)^2$,  i.e.\ at the boundary deflection $w = 0$ and rotation $\thetav = 0$ vanish. The plate thickness varies from $t=0.1$ to $t= 10^{-5}$. Young's modulus and Poisson ratio are chosen as $E = 12$ and $\nu = 0$. The shear correction factor is set to $k_s = 5/6$. The vertical component of the volume load is chosen as
\begin{eqnarray}
\lefteqn{
    f_z(x, y) =}\nonumber\\
    &&\frac{E}{1-\nu^2} \Big( y(y-1)(5x^2-5x+1)\big(2y^2(y-1)^2+x(x-1)(5y^2-5y+1)\big) \nonumber\\
    && + x(x-1)(5y^2-5y+1)\big(2x^2(x-1)^2+y(y-1)(5x^2-5x+1)\big) \Big).
\end{eqnarray}
The solution $(\thetav,w)$ is given by
\begin{eqnarray}
\theta_1(x,y) &=& y^3(y-1)^3x^2(x-1)^2(2x-1), \\
\theta_2(x,y) &=& x^3(x-1)^3y^2(y-1)^2(2y-1), \\
w(x,y) &=& \frac{1}{3} x^3(x-1)^3y^3(y-1)^3 \nonumber\\
&& - \frac{2t^2}{5(1-\nu)} \big(y^3(y-1)^3x(x-1)(5x^2-5x+1) \nonumber\\
&&\qquad\qquad\quad + x^3(x-1)^3y(y-1)(5y^2-5y+1) \big).
\end{eqnarray}

Two discretization methods are compared: the MITC7 element \cite{BrezziBatheFortin:89} and the TDNNS element for $k=1$ and $k=2$. In case of the MITC7 element and the TDNNS element with $k=1$, the deflections are discretized by polynomials of order two. For the higher-order TDNNS element, the deflections are of order three.

First, we compare the different methods for a thickness of $t=10^{-3}$. In Figure~\ref{fig:clampedplate_w}, the convergence of $\|w-w_h\|_{L^2(\Omega)}$ is shown, Figure~\ref{fig:clampedplate_theta} displays the convergence of $\|\thetav -\thetav_h\|_{\Lv^2(\Omega)}$. It shows that for the deflection $w$, both the MITC7 element and the lowest-order TDNNS element with $k=1$ show convergence order three, while the  TDNNS element with $k=2$ converges, as expected, at order four. However, for the rotations $\theta$, the MITC7 element and the TDNNS element with $k=2$ converge at the same rate of order three, while the lowest-order TDNNS element shows a convergence rate of order two. Thus, from the point of view of convergence, the MITC7 element lies in between the TDNNS elements with $k=1$ and $k=2$.

\begin{figure}
    \begin{center}
        \includegraphics[angle=270,width=0.8\columnwidth]
            {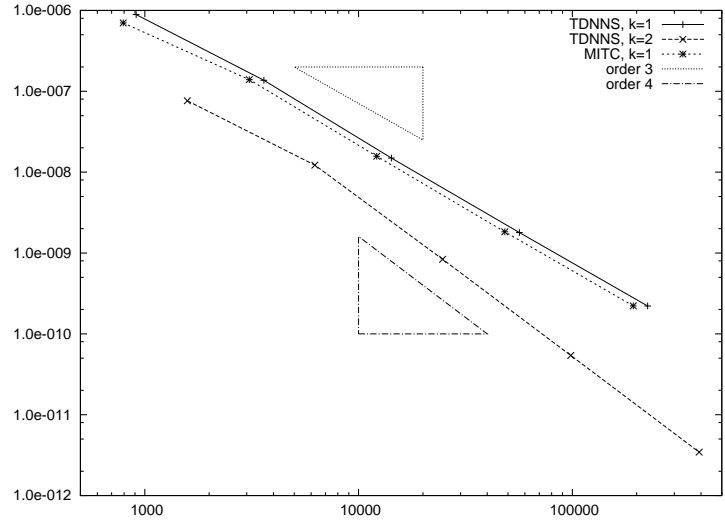}
    \end{center}
    \caption{Convergence of $\|w-w_h\|_{L^2}$ over the number of degrees of freedom for the MITC7 element as well as the
                        TDNNS element for $k=1$ and $k=2$, thickness $t=10^{-3}$.}
    \label{fig:clampedplate_w}
\end{figure}

\begin{figure}
    \begin{center}
        \includegraphics[angle=270,width=0.8\columnwidth]
            {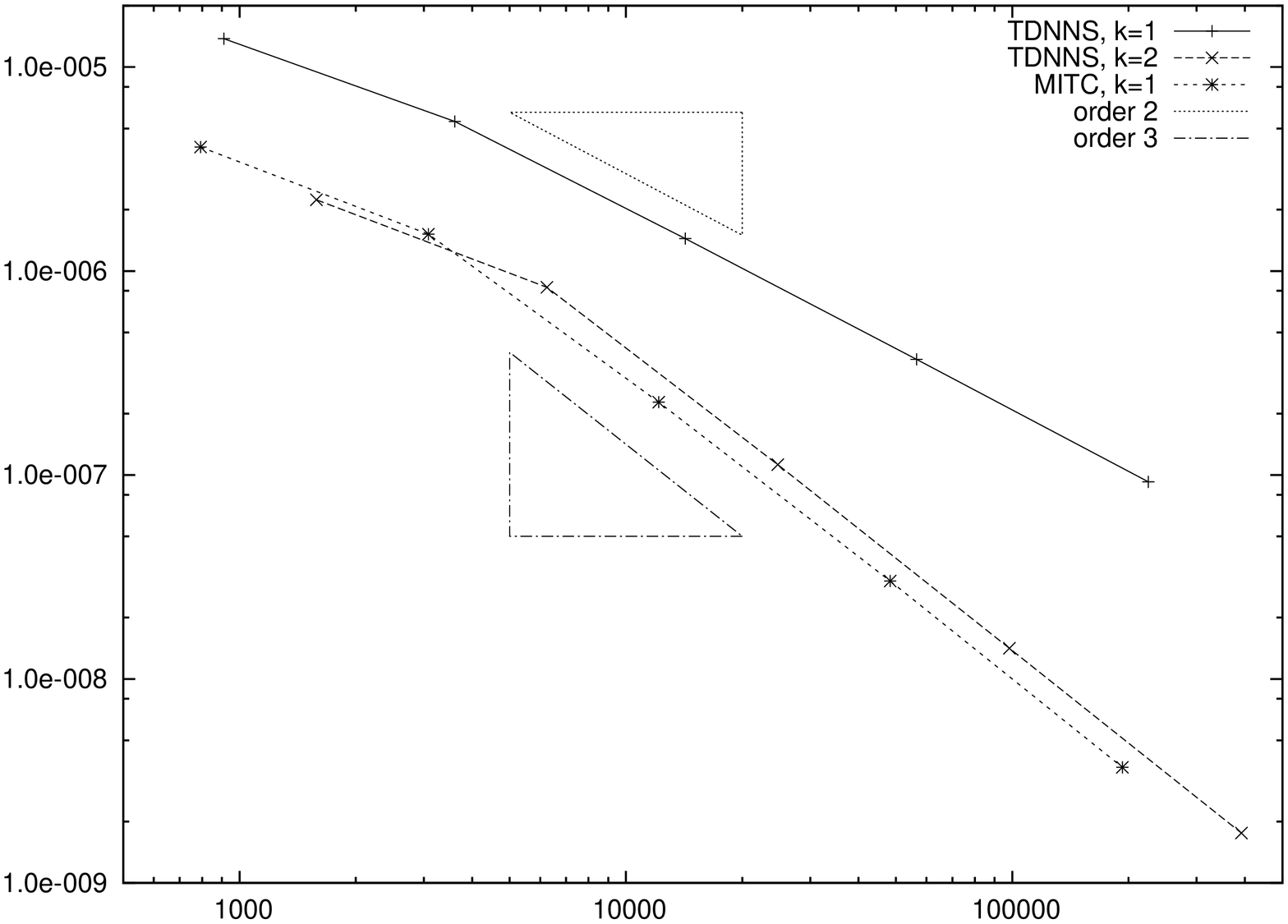}
    \end{center}
    \caption{Convergence of $\|\thetav-\thetav_h\|_{\Lv^2}$ over the number of degrees of freedom for the MITC7 element as well as the
                        TDNNS element for $k=1$ and $k=2$, thickness $t=10^{-3}$.}
    \label{fig:clampedplate_theta}
\end{figure}

Next, we plot the convergence of the lowest order TDNNS method for different thicknesses. Figure~\ref{fig:clampedplate_comparet} shows the convergence of the method for thicknesses varying between $t=0.1$ and $t=10^{-5}$. The error curves are very close, as the method does not suffer from the degrading thickness.
\begin{figure}
    \begin{center}
        \includegraphics[angle=270,width=0.8\columnwidth]
            {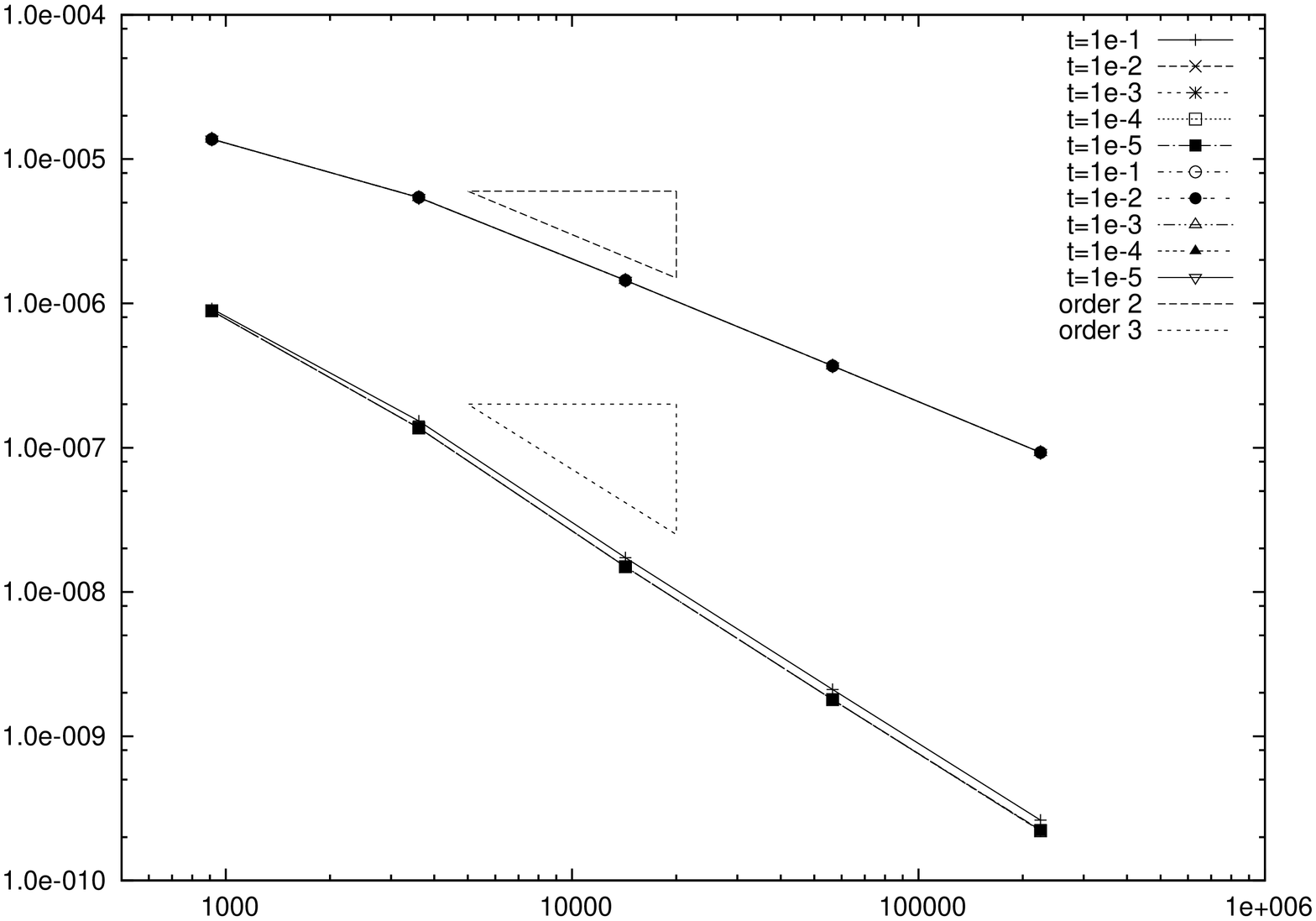}
    \end{center}
    \caption{Convergence of $\|w-w_h\|_{L^2}$ and $\|\thetav-\thetav_h\|_{\Lv^2}$ over the number of degrees of freedom for the
                        TDNNS element for $k=1$, thickness $t$ varying from $t=0.1$ to $t=10^{-5}$.}
    \label{fig:clampedplate_comparet}
\end{figure}

\subsection{Square plate with hole}

In the second example, we consider a square plate of dimensions $100 \times 100$~mm, in which a circular hole of diameter $d = 30$~mm is cut. The Young's modulus $E = 2.1\times10^5$~N/mm$^2$ and Poisson ratio $\nu = 0.3$ are chosen as those of steel. The shear correction factor is set to $k_s = 5/6$. The plate is clamped at the left hand side ($x=0$), and a surface traction $\sigma_{zz} = 0.1(y-50)$~N/mm$^2$ acts on the right hand side. All other boundaries are free. See Figure~\ref{fig:plate_hole_sketch} for a sketch of the setup.

An initial mesh consisting of 56 elements of mesh size approximately $h = 30$~mm is used. A two-level geometric refinement towards the corners and the free boundary at the center hole is applied to catch singularities, leading to a total number of 95 elements. The TDNNS method with $k=4$ is applied, which leads to 2593 coupling degrees of freedom.

The bending moments $\Mt_{xy}$ and $\Mt_{yy}$ are depicted in Figure~\ref{fig:plate_hole_mxy} and Figure~\ref{fig:plate_hole_myy}, respectively. Note that in Figure~\ref{fig:plate_hole_mxy}, different scales are used for the original plate and the zoom to the interior hole, such that the
steep gradient of the bending moment  becomes visible.

\begin{figure}
    \begin{center}
        \includegraphics[width=0.8\columnwidth]
            {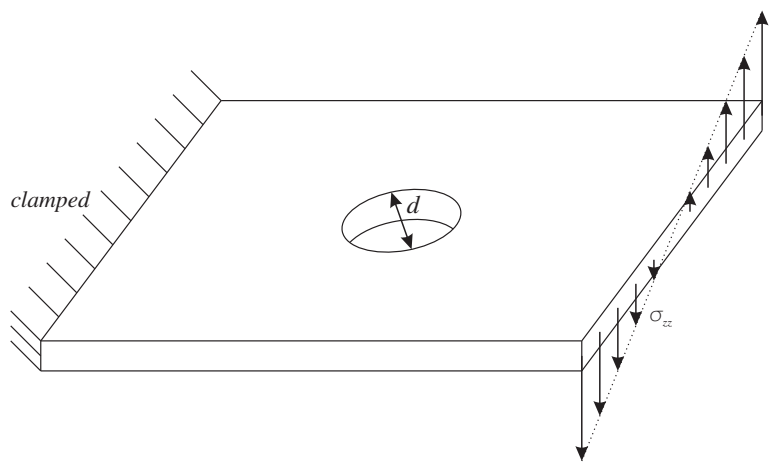}
    \end{center}
    \caption{Sketch of the setup for the plate with hole.}
    \label{fig:plate_hole_sketch}
\end{figure}

\begin{figure}
    \begin{center}
        \includegraphics[width=0.8\columnwidth]
            {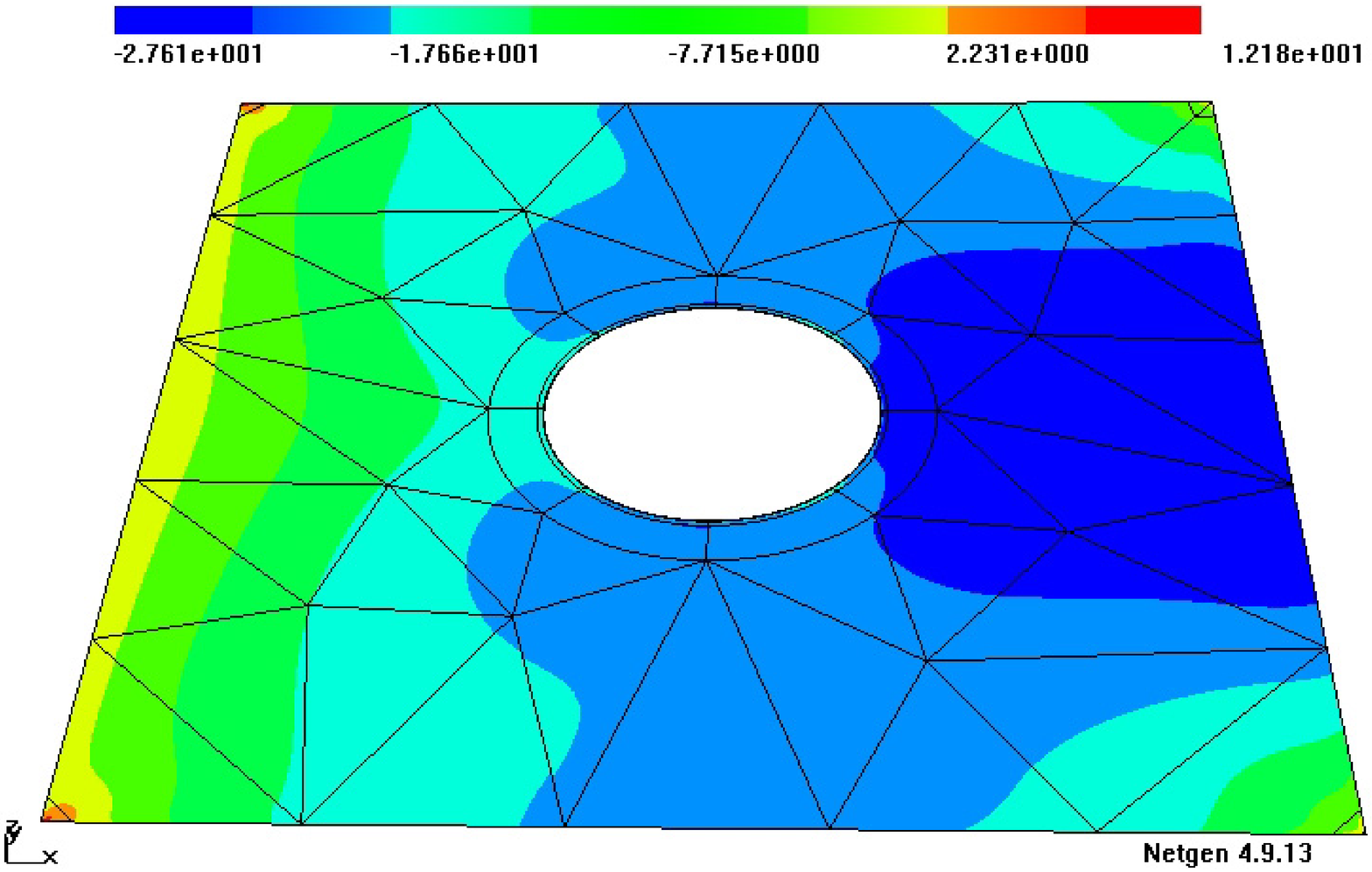}
        \includegraphics[width=0.8\columnwidth]
            {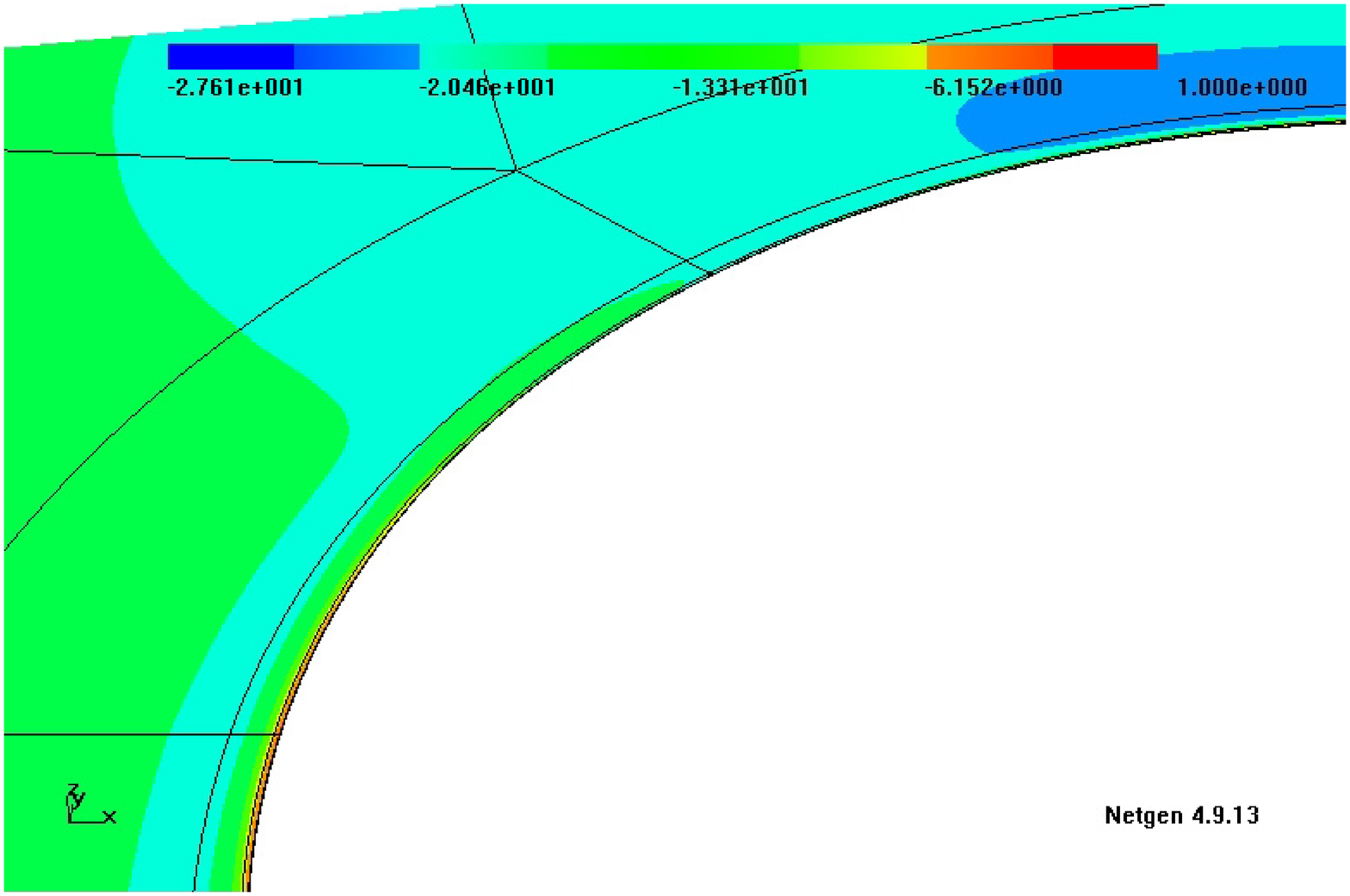}
    \end{center}
    \caption{Bending moment $\Mt_{xy}$, using TDNNS elements with $k=4$. A different scale is used in the zoom to the interior hole to make visible the steep gradient of the bending moment.}
    \label{fig:plate_hole_mxy}
\end{figure}

\begin{figure}
    \begin{center}
        \includegraphics[width=0.8\columnwidth]
            {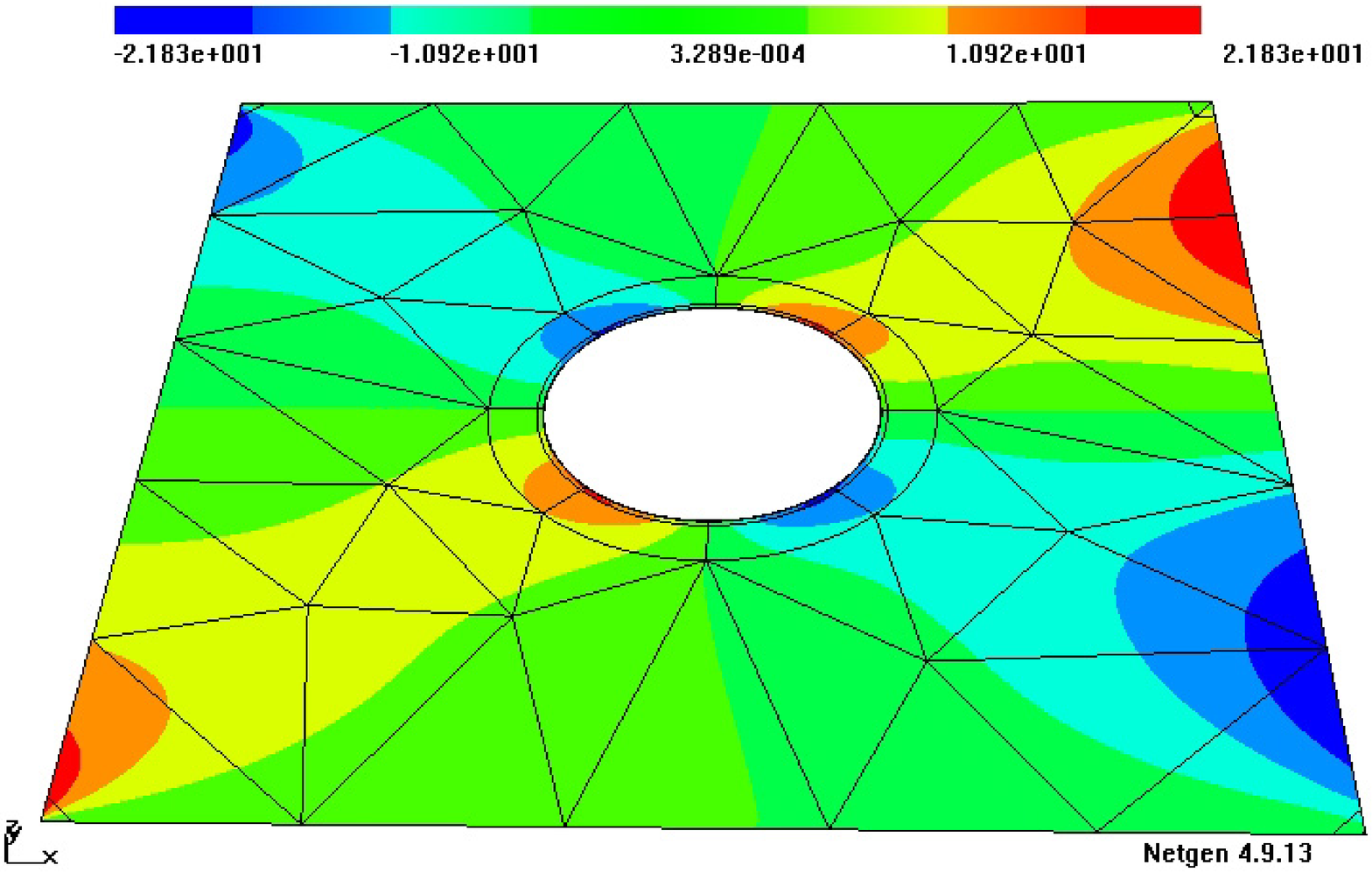}
        \includegraphics[width=0.8\columnwidth]
            {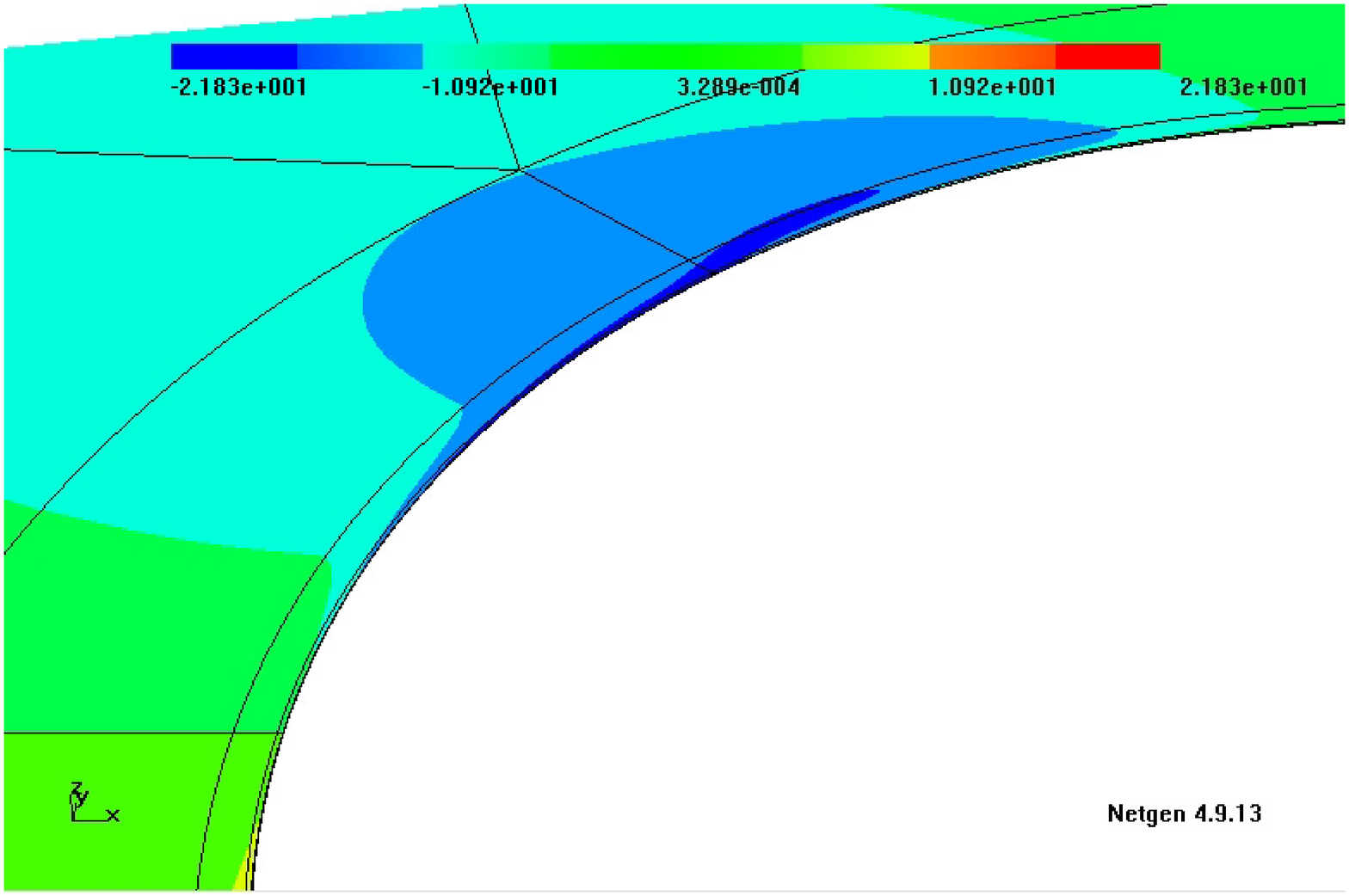}
    \end{center}
    \caption{Bending moment $\Mt_{yy}$, using TDNNS elements with $k=4$.}
    \label{fig:plate_hole_myy}
\end{figure}

\bibliographystyle{plain}      
\bibliography{PechsteinSchoeberl_RM_bib_revised}   

\begin{thebibliography}{10}

\bibitem{AdamsCockburn:05}
S.~Adams and B.~Cockburn.
\newblock A mixed finite element method for elasticity in three dimensions.
\newblock {\em J. Sci. Comput.}, 25(3):515--521, 2005.

\bibitem{Arnold:81}
D.~N. Arnold.
\newblock Discretization by finite elements of a model parameter dependent
  problem.
\newblock {\em Numer. Math.}, 37(3):405--421, 1981.

\bibitem{ArnoldAwanouWinther:07}
D.~N. Arnold, G.~Awanou, and R.~Winther.
\newblock Finite elements for symmetric tensors in three dimensions.
\newblock {\em Math. Comp.}, 77(263):1229--1251, 2008.

\bibitem{ArnoldBrezziDouglas:84}
D.~N. Arnold, F.~Brezzi, and J.~Douglas.
\newblock {PEERS:} a new finite element for plane elasticity.
\newblock {\em Jap. J. Appl. Math.}, 1:347--367, 1984.

\bibitem{ArnoldBrezziFalkMarini:07}
D.~N. Arnold, F.~Brezzi, R.~S. Falk, and L.~D. Marini.
\newblock Locking-free {R}eissner-{M}indlin elements without reduced
  integration.
\newblock {\em Comput. Methods Appl. Mech. Engrg.}, 196(37-40):3660--3671,
  2007.

\bibitem{ArnoldFalk:89}
D.~N. Arnold and R.~S. Falk.
\newblock A uniformly accurate finite element method for the
  {R}eissner-{M}indlin plate.
\newblock {\em SIAM J. Numer. Anal.}, 26(6):1276--1290, 1989.

\bibitem{ArnoldFalkWinther:07}
D.~N. Arnold, R.~S. Falk, and R.~Winther.
\newblock Mixed finite element methods for linear elasticity with weakly
  imposed symmetry.
\newblock {\em Math. Comp.}, 76(260):1699--1723 (electronic), 2007.

\bibitem{ArnoldWinther:02}
D.~N. Arnold and R.~Winther.
\newblock Mixed finite elements for elasticity.
\newblock {\em Numer. Math.}, 92(3):401--419, 2002.

\bibitem{AuricchioLovadina:01}
F.~Auricchio and C.~Lovadina.
\newblock Analysis of kinematic linked interpolation methods for
  {R}eissner-{M}indlin plate problems.
\newblock {\em Comput. Methods Appl. Mech. Engrg.}, 190(18-19):2465--2482,
  2001.

\bibitem{BatheDvorkin:85}
K.-J. Bathe and E.~N. Dvorkin.
\newblock A four node plate bending element based on {M}indlin/{R}eissner plate
  theory and a mixed interpolation.
\newblock {\em International Journal for Numerical Methods in Engineering},
  21:367--383, 1985.

\bibitem{BehrensGuzman:11}
E.~M. Behrens and J.~Guzm{\'a}n.
\newblock A new family of mixed methods for the {R}eissner-{M}indlin plate
  model based on a system of first-order equations.
\newblock {\em J. Sci. Comput.}, 49(2):137--166, 2011.

\bibitem{VeigaMoraRodriguez:13}
L.~{Beirao da Veiga}, D.~Mora, and R.~Rodr{\'{\i}}guez.
\newblock Numerical analysis of a locking-free mixed finite element method for
  a bending moment formulation of {R}eissner-{M}indlin plate model.
\newblock {\em Numer. Methods Partial Differential Equations}, 29(1):40--63,
  2013.

\bibitem{BoffiBrezziFortin:13}
D.~Boffi, F.~Brezzi, and M.~Fortin.
\newblock {\em Mixed finite element methods and applications}, volume~44 of
  {\em Springer Series in Computational Mathematics}.
\newblock Springer, Heidelberg, 2013.

\bibitem{BoesingMadureiraMozolevski:10}
P.~R. B{\"o}sing, A.~L. Madureira, and I.~Mozolevski.
\newblock A new interior penalty discontinuous {G}alerkin method for the
  {R}eissner-{M}indlin model.
\newblock {\em Math. Models Methods Appl. Sci.}, 20(8):1343--1361, 2010.

\bibitem{BrennerScott:02}
S.~C. Brenner and L.~R. Scott.
\newblock {\em The Mathematical Theory of Finite Element Methods}.
\newblock Springer-Verlag, New York, 2002.

\bibitem{BrezziBatheFortin:89}
F.~Brezzi, K.-J. Bathe, and M.~Fortin.
\newblock Mixed-interpolated elements for {R}eissner-{M}indlin plates.
\newblock {\em International Journal for Numerical Methods in Engineering},
  28(8):1787--1801, 1989.

\bibitem{BrezziFortinStenberg:91}
F.~Brezzi, M.~Fortin, and R.~Stenberg.
\newblock Error analysis of mixed-interpolated elements for
  {R}eissner-{M}indlin plates.
\newblock {\em Math. Models Methods Appl. Sci.}, 1(2):125--151, 1991.

\bibitem{BrezziMarini:03}
F.~Brezzi and L.D. Marini.
\newblock A nonconforming element for the {R}eissner-{M}indlin plate.
\newblock {\em Computers \& Structures}, 81(8-11):515 -- 522, 2003.

\bibitem{CaloCollierNiemi:14}
V.~M. Calo, N.~O. Collier, and A.~H. Niemi.
\newblock Analysis of the discontinuous {P}etrov-{G}alerkin method with optimal
  test functions for the {R}eissner-{M}indlin plate bending model.
\newblock {\em Comput. Math. Appl.}, 66(12):2570--2586, 2014.

\bibitem{CarstensenXieYuZhou:11}
C.~Carstensen, X.~Xie, G.~Yu, and T.~Zhou.
\newblock A priori and a posteriori analysis for a locking-free low order
  quadrilateral hybrid finite element for {R}eissner-{M}indlin plates.
\newblock {\em Comput. Methods Appl. Mech. Engrg.}, 200(9-12):1161--1175, 2011.

\bibitem{CastellazziKrysl:09}
G.~Castellazzi and P.~Krysl.
\newblock {Displacement-based finite elements with nodal integration for
  {R}eissner-{M}indlin plates}.
\newblock {\em {International Journal for Numerical Methods in Engineering}},
  {80}:{135--62}, {8 Oct.} {2009}.

\bibitem{ChapelleStenberg:98}
D.~Chapelle and R.~Stenberg.
\newblock An optimal low-order locking-free finite element method for
  {R}eissner-{M}indlin plates.
\newblock {\em Math. Models Methods Appl. Sci.}, 8(3):407--430, 1998.

\bibitem{ChinosiLovadinaMarini:06}
C.~Chinosi, C.~Lovadina, and L.~D. Marini.
\newblock Nonconforming locking-free finite elements for {R}eissner-{M}indlin
  plates.
\newblock {\em Comput. Methods Appl. Mech. Engrg.}, 195(25-28):3448--3460,
  2006.

\bibitem{Comodi:89}
M.~I. Comodi.
\newblock The {H}ellan-{H}errmann-{J}ohnson method: some new error estimates
  and postprocessing.
\newblock {\em Math. Comp.}, 52(185):17--29, 1989.

\bibitem{FalkTu:99}
R.~S. Falk and T.~Tu.
\newblock Locking-free finite elements for the {R}eissner-{M}indlin plate.
\newblock {\em Math. Comp.}, 69(231):911--928, 2000.

\bibitem{GruttmannWagner:04}
F.~Gruttmann and W.~Wagner.
\newblock {A stabilized one-point integrated quadrilateral {R}eissner-{M}indlin
  plate element}.
\newblock {\em {International Journal for Numerical Methods in Engineering}},
  {61}({13}):{2273--2295}, {DEC 7} {2004}.

\bibitem{HansboHeintzLarson:11}
P.~Hansbo, D.~Heintz, and M.~G. Larson.
\newblock A finite element method with discontinuous rotations for the
  {M}indlin-{R}eissner plate model.
\newblock {\em Comput. Methods Appl. Mech. Engrg.}, 200(5-8):638--648, 2011.

\bibitem{Hellan:67}
K{\aa}re Hellan.
\newblock Analysis of elastic plates in flexure by a simplified finite element
  method.
\newblock {\em Acta Polytechnica Scandinavica - Civil Engineering And Building
  Construction Series}, (46):1, 1967.

\bibitem{Herrmann:67}
Leonard~R Herrmann.
\newblock Finite-element bending analysis for plates.
\newblock {\em Journal of the Engineering Mechanics Division}, 93(5):13--26,
  1967.

\bibitem{HughesFranca:88}
T.~J.~R. Hughes and L.~P. Franca.
\newblock A mixed finite element formulation for {R}eissner-{M}indlin plate
  theory: uniform convergence of all higher-order spaces.
\newblock {\em Comput. Methods Appl. Mech. Engrg.}, 67(2):223--240, 1988.

\bibitem{HughesTezduyar:81}
T.~J.~R. Hughes and T.~E. Tezduyar.
\newblock {Finite elements based upon {M}indlin plate theory with particular
  reference to the four-node bilinear isoparametric element}.
\newblock {\em {Transactions of the ASME. Journal of Applied Mechanics}},
  {48}:{587--96}, {Sept.} {1981}.

\bibitem{Johnson:73}
Claes Johnson.
\newblock On the convergence of a mixed finite-element method for plate bending
  problems.
\newblock {\em Numer. Math.}, 21:43--62, 1973.

\bibitem{MacNeal:78}
R.~H. MacNeal.
\newblock A simple quadrilateral shell element.
\newblock {\em Computers \& Structures}, 8(2):175 -- 183, 1978.

\bibitem{Mindlin:51}
R.~D. Mindlin.
\newblock Influence of rotatory inertia and shear flexural motions of isotropic
  elastic plates.
\newblock {\em Journal of Applied Mechanics}, 18:31--38, 1951.

\bibitem{Monk:03}
P.~Monk.
\newblock {\em Finite element methods for {M}axwell's equations}.
\newblock Numerical Mathematics and Scientific Computation. Oxford University
  Press, New York, 2003.

\bibitem{Nedelec:80}
J.~C. N{\'e}d{\'e}lec.
\newblock Mixed finite elements in $\mathbb {R}^3$.
\newblock {\em Numerische Mathematik}, 35:315--341, 1980.

\bibitem{Nedelec:86}
J.~C. N{\'e}d{\'e}lec.
\newblock A new family of mixed finite elements in $\mathbb {R}^3$.
\newblock {\em Numerische Mathematik}, 50:57--81, 1986.

\bibitem{PechsteinSchoeberl:11}
A.~Pechstein and J.~Sch{\"o}berl.
\newblock Tangential-displacement and normal-normal-stress continuous mixed
  finite elements for elasticity.
\newblock {\em Math. Models Methods Appl. Sci.}, 21(8):1761--1782, 2011.

\bibitem{PechsteinSchoeberl:11anis}
A.~Pechstein and J.~Sch{\"o}berl.
\newblock Anisotropic mixed finite elements for elasticity.
\newblock {\em International Journal for Numerical Methods in Engineering},
  90(2):196--217, 2012.

\bibitem{PechsteinSchoeberl:NN}
A.~S. {Pechstein} and J.~{Sch{\"o}berl}.
\newblock {An analysis of the {T}{D}{N}{N}{S} method using natural norms}.
\newblock {\em ArXiv e-prints}, June 2016.

\bibitem{PontazaReddy:04}
J.~P. Pontaza and J.~N. Reddy.
\newblock {Mixed plate bending elements based on least-squares formulation}.
\newblock {\em {International Journal for Numerical Methods in Engineering}},
  {60}({5}):{891--922}, {JUN 7} {2004}.

\bibitem{Reissner:45}
E.~Reissner.
\newblock The effect of transverse shear deformation on the bending of elastic
  plates.
\newblock {\em Journal of Applied Mechanics}, 12:69--76, 1945.

\bibitem{Sinwel:09}
A.~Sinwel.
\newblock {\em A New Family of Mixed Finite Elements for Elasticity}.
\newblock PhD thesis, Johannes Kepler University Linz, 2009.
\newblock Published by {S}\"udwestdeutscher Verlag f\"ur Hochschulschriften,
  June 2009.

\bibitem{Stenberg:86}
R.~Stenberg.
\newblock On the construction of optimal mixed finite element methods for the
  linear elasticity problem.
\newblock {\em Numer. Math.}, 42:447--462, 1986.

\bibitem{Stenberg:88}
R.~Stenberg.
\newblock A family of mixed finite elements for the elasticity problem.
\newblock {\em Numer. Math. 53, 513-538}, 1988.

\bibitem{StenbergSuri:97}
R.~Stenberg and M.~Suri.
\newblock An {$hp$} error analysis of {MITC} plate elements.
\newblock {\em SIAM J. Numer. Anal.}, 34(2):544--568, 1997.

\bibitem{TaylorAuricchio:93}
R.~L. Taylor and F.~Auricchio.
\newblock Linked interpolation for {R}eissner-{M}indlin plate elements: Part
  {II}: A simple triangle.
\newblock {\em International Journal for Numerical Methods in Engineering},
  36(18):3057--3066, 1993.

\bibitem{ZienkiewiczTaylor:71}
O.~C. Zienkiewicz, R.~L. Taylor, and J.~M. Too.
\newblock Reduced integration technique in general analysis of plates and
  shells.
\newblock {\em International Journal for Numerical Methods in Engineering},
  3(2):275--290, 1971.

\bibitem{ZienkiewiczEtAl:93}
O.~C. Zienkiewicz, Z.~Xu, L.~F. Zeng, A.~Samuelsson, and N.-E. Wiberg.
\newblock Linked interpolation for {R}eissner-{M}indlin plate elements: Part
  {I}: A simple quadrilateral.
\newblock {\em International Journal for Numerical Methods in Engineering},
  36(18):3043--3056, 1993.

\end{thebibliography}

\end{document}